\documentclass{article}
\usepackage{amsmath}
\usepackage{amssymb}
\usepackage{amsfonts}
\usepackage{latexsym,longtable}
\usepackage{graphicx}
\usepackage{tikz}
\usetikzlibrary{matrix}
\begin{document}
\newcommand{\B}{{\cal B}}
\newcommand{\D}{{\cal D}}
\newcommand{\E}{{\cal E}}
\newcommand{\F}{{\cal F}}
\newcommand{\A}{{\cal A}}
\newcommand{\Hh}{{\cal H}}
\newcommand{\Pp}{{\cal P}}
\newcommand{\Z}{{\bf Z}}
\newcommand{\T}{{\cal T}}
\newcommand{\ZZ}{{\mathbb{Z}}}
\newcommand{\qed}{\hphantom{.}\hfill $\Box$\medbreak}
\newcommand{\proof}{\noindent{\bf Proof \ }}
\renewcommand{\theequation}{\thesection.\arabic{equation}}
\newtheorem{theorem}{Theorem}[section]
\newtheorem{lemma}[theorem]{Lemma}
\newtheorem{corollary}[theorem]{Corollary}
\newtheorem{remark}[theorem]{Remark}
\newtheorem{example}[theorem]{Example}
\newtheorem{definition}[theorem]{Definition}
\newtheorem{construction}[theorem]{Construction}


\medskip
\title{New Necessary Conditions for Existence of Strong External Difference Families \thanks{Research supported by NSFC grants 11701303 (J. Bao), 12271390 (L. Ji).   {\em Corresponding author: Jingjun Bao.}
 }}

 \author{{\small   Jingjun Bao$^1$, \ Lijun Ji$^2$} \\
 {\small $^1$  The School of Mathematics and Statistics, Ningbo University, Ningbo 315211, China}\\
 {\small $^2$ Department of Mathematics, Soochow University, Suzhou, 215006, China}\\
 {\small E-mail: baojingjun@hotmail.com,\ jilijun@suda.edu.cn}\\
 }


\date{}
\maketitle
\begin{abstract}
\noindent \\ \noindent \\
Strong external difference families (SEDFs) were introduced by Paterson and Stinson as a more restrictive version of external difference families. SEDFs can be used to produce optimal strong algebraic manipulation detection codes. In this paper, we use the theory of cyclotomic fields, algebraic number theory and character theory to give some new necessary conditions for the existence of SEDFs. Based on the results of decomposition of prime ideals and Schmidt's field descent method, two exponent bounds of SEDFs are presented. Based on the field descent method, a special homomorphism from an abelian group to its cyclic subgroup and Gauss sums, some bounds for prime divisors of $v$ and some congruence relations between $k, m$ and $\lambda$ for $(v,m,k,\lambda)$-SEDFs with $m>2$ are established.

\medskip


\medskip


\end{abstract}


\section{Introduction}

Motivated by applications to algebraic manipulation detection codes (or AMD codes) \cite{CDFPW2008, CFP2013, CPX2015}, Paterson and Stinson introduced strong external difference families (or SEDFs) in \cite{PS2016}. SEDFs are closely related to but stronger than  external difference families (or EDFs) \cite{OKSS2004}. In \cite{PS2016}, it was noted that optimal AMD codes can be obtained from EDFs, whereas optimal strong AMD codes can be obtained from SEDFs. See \cite{PS2016} for a discussion of these and related structures and how they relate to AMD codes.

In 2016, Martin and Stinson \cite{MS-arXiv} and Hucznska and Paterson \cite{HP2018} further investigated the existence of SEDFs and gave some new nonexistence results of certain SEDFs. Their results showed that the existence of SEDFs is an interesting mathematical problem in its own right, independent of any applications to AMD codes. In \cite{BJWZ2018} and \cite{MS-arXiv}, character theory was used to obtain some necessary conditions for the existence of SEDFs and some nonexistence results of SEDFs were presented. In \cite{JL2018}, Jedwab and Li characterized the parameters $(v,m,k,\lambda)$ of a nontrivial near-complete SEDF and gave a range of nonexistence results by using character theory and algebraic number theory. In \cite{LNC2018}, difference sets and cyclotomic classes were used to construct families of generalized SEDFs.  In \cite{HJN2019}, Huczynska et al. presented the first family of non-abelian SEDFs and gave some necessary conditions of SEDFs. In \cite{LLP2019}, Leung et al. developed various  nonexistence results of SEDFs in an abelian group $G$, where $|G|$ is a product of at most three not necessarily distinct primes. In \cite{LP2022}, Leung et al. derived some new nonexistence results for $(v,m,k,pq)$-SEDFs in an abelian group $G$, where $p$ and $q$ are primes.
Let us recall the definition of SEDFs.

Let $G$ be a finite abelian group of order $v$ (written multiplicatively) with the identity element $e\in  G$. For any two nonempty subsets $D_1,D_2$ of $G$, the multiset
$$\Delta_E(D_1,D_2)=\{xy^{-1}\colon x\in D_1, y\in D_2\}$$
is called the external difference of $D_1$ and $D_2$.
Let $k, \lambda, m$ be positive integers and let $D_1,\ldots, D_m$ be pairwise disjoint $k$-subsets of $G$. If the following multiset equation holds:
\begin{equation}
\label{SEDF11}
\bigcup\limits_{\{\ell \colon \ell\neq j\}}\Delta_E(D_{\ell},D_j)=\lambda (G\setminus \{e\})
\end{equation}
for each $1\leq j\leq m$,
then the collection $\{D_1,\ldots, D_m\}$ is called a {\bf strong external difference family} and denoted as a $(v,m,k,\lambda)$-SEDF.

When the multiset equation is replaced with $$\bigcup\limits_{\{\ell,j \colon \ell\neq j\}}\Delta_E(D_{\ell},D_j)=\lambda (G\setminus \{e\}),$$ the collection $\{D_1,\ldots, D_m\}$ is called an {\bf external difference family} and denoted as a $(v, m, k, \lambda)$-EDF. Clearly, a $(v,m,k,\lambda)$-SEDF is a $(v,m,k,m\lambda)$-EDF.

From the definition of an SEDF, it is easy to see that $m>1$, $mk\leq v$ and
\begin{equation}
\label{mk2=lambda(n-1)}
(m-1)k^2=\lambda (v-1).
\end{equation}
A $(v, m, k, \lambda)$-SEDF is trivial if $k=1$; it follows from (\ref{mk2=lambda(n-1)}) that the parameters of a trivial SEDF have the form $(v,v,1,1)$, and an SEDF with these parameters exists (trivially) in every group of order $v$. The following lemma describes the parameters and groups of the known nontrivial SEDFs.

\begin{lemma}\label{SEDFL1}
A $(v,m,k,\lambda)$-SEDF in $G$ exists in each of the following cases:

{\rm (1)} $(v,m,k,\lambda)=(k^2+1, 2,k,1)$ and $G=\mathbb{Z}_{k^2+1}$ {\rm \cite[Example 2.2]{PS2016}}.

{\rm (2)} $(v,m,k,\lambda)=(v, 2,\frac{v-1}{2},\frac{v-1}{4})$ and $v\equiv 1\pmod 4$, provided that there exists a
$(v, \frac{v-1}{2},\frac{v-5}{2},\frac{v-1}{4})$ partial difference set in $G$  {\rm \cite[Theorem 4.4]{HP2018}}.

{\rm (3)} $(v,m,k,\lambda)=(p, 2,\frac{p-1}{4},\frac{p-1}{16})$ where $p=16t^2+1$ is a prime and $t$ is an integer, and $G=\mathbb{Z}_{p}$ {\rm \cite[Theorem 4.3]{BJWZ2018}}.

{\rm (4)} $(v,m,k,\lambda)=(p, 2,\frac{p-1}{6},\frac{p-1}{36})$ where $p=108t^2+1$ is a prime and $t$ is an integer, and $G=\mathbb{Z}_{p}$ {\rm \cite[Theorem 4.6]{BJWZ2018}}.

{\rm (5)} $(v,m,k,\lambda)=(243, 11, 22, 20)$ and $G=\mathbb{F}_{3^5}$ {\rm \cite[Theorem 3.1]{JL2018}}, {\rm \cite[Section 2]{WYFF2018}}.

\end{lemma}

When $\lambda=1$, the parameters of a nontrivial $(v,m,k,\lambda)$-SEDF have been characterized.

\begin{lemma} {\rm \cite{PS2016}}
There exists a $(v, m, k, 1)$-SEDF if and only if $m=2$ and $v=k^2+1$.
\end{lemma}

By using character theory, Martin and Stinson \cite{MS-arXiv} established various nonexistence results of SEDFs, including the following:

\begin{lemma} {\rm\cite{MS-arXiv}}
\label{m34}
 If there is a $(v,m,k,\lambda)$-SEDF with $k>1$, then $m\neq 3$ and $m\neq 4$.

\end{lemma}

\begin{lemma} {\rm\cite{MS-arXiv}}
\label{t.stinson}
If $G$ is any group of prime order, $k > 1$ and $m > 2$, then there does not exist a $(v, m, k, \lambda)$-SEDF in $G$.
\end{lemma}

Huczynska and Paterson \cite{HP2018} gave the following necessary conditions for the existence of SEDFs with $m \geq 3$.

\begin{lemma} {\rm\cite{HP2018}}
\label{NC2}
Suppose there exists a $(v, m, k, \lambda)$-SEDF with $m\geq 3$ and $k> \lambda \geq 2$. Then the
following inequality must hold: $$\frac{\lambda(k-1)(m-2)}{(\lambda-1)k(m-1)}\leq 1.$$
\end{lemma}

Bao et al. \cite{BJWZ2018} gave some nonexistence results of SEDFs by using character theory.

\begin{lemma} {\rm\cite{BJWZ2018}}
\label{NC2}
A $(v,m,k,\lambda)$-SEDF in $G$ does not exist in each of the following cases:

{\rm (1)} $m>4$ and $v \equiv 0\pmod k$.

{\rm (2)} $m>4$ and for any prime divisor $p$ of $v$, $m \not\equiv 2\pmod p$ whenever $\gcd(mk,p)=1$.


{\rm (3)} $(v,m,k,\lambda)=(p^2, m,k,\lambda)$, where $p$ is a prime, $m>2$ and $G=\mathbb{Z}_{p^2}$.

\end{lemma}

Recently, Huczynska et al. \cite{HJN2019} obatined the following nonexistence results of SEDFs.

\begin{lemma} {\rm\cite{HJN2019}}
\label{NC2}
Let $v-1$ be square-free. Then there exists no non-trivial $(v,m,k,\lambda)$-SEDF.
\end{lemma}

In \cite{LLP2019}, Leung et al. developed some new nonexistence results of SEDFs as follows.

\begin{lemma} {\rm\cite{LLP2019}}
\label{NC2}
A $(v,m,k,\lambda)$-SEDF in an abelian $G$ does not exist in each of the following cases:

{\rm (1)} $(v,m,k,\lambda)=(pqr, m,k,\lambda)$, where $p, q, r$ are three distinct primes, $m>2$ and $G=\mathbb{Z}_{pqr}$.

{\rm (2)} $(v,m,k,\lambda)=(pq^n, m,k,\lambda)$, where $p, q$ are two distinct primes, $m>2$ and $n\in \{1,2\}$.

{\rm (3)} $(v,m,k,\lambda)=(p^n, m,k,\lambda)$, where $p$ is a prime, $m>2$ and $G=\mathbb{Z}_{p^n}$.

{\rm (4)} $(v,m,k,\lambda)=(p^3, m,k,\lambda)$, where $p$ is a prime, $m>2$ and $G=\mathbb{Z}_{p}\times\mathbb{Z}_{p^2}$.

{\rm (5)} $(v,m,k,\lambda)=(p^3, m,k,\lambda)$ with $m>2$ and $p$ being a prime less than $3\times 10^{12}$, where $G=\mathbb{Z}_{p}\times\mathbb{Z}_{p}\times\mathbb{Z}_{p}$.

\end{lemma}

Very recently, Leung et al. derived some new nonexistence results of SEDFs.

\begin{lemma} {\rm\cite{LP2022}}
\label{NC20}
A $(v,m,k,\lambda)$-SEDF in an abelian $G$ does not exist in each of the following cases:

{\rm (1)} $(v,m,k,\lambda)=(v, m,k, p^2)$, where $p$ is a prime and $m>2$.

{\rm (2)} $(v,m,k,\lambda)=(v, m,k, 2p)$, where $p$ is a prime and $m>2$.

{\rm (3)} $(v,m,k,\lambda)=(v, m,k, pq)$, where  $m>2$ and $p,q$ are primes with $q<p<\frac{q^3}{2}$ and $q<200$.

{\rm (4)} $(v,m,k,\lambda)=(v, m,k, pq)$, where  $m>2$ and $p,q$ are primes with $q<p$ and $q=3, 5, 7, 13, 19, 31$.

\end{lemma}

From now on, we shall use $G_p$ to denote the Sylow $p$-subgroup of the group $G$, where $p$ is a prime. Let $g\in G$, we use $\text{ord}(g)$ to denote the order of $g$. The {\bf exponent} of a finite abelian group $G$ is defined as the least common multiple of the orders of all elements of the group $G$. We use $ \text{exp}(G)$ to denote the exponent of the group $G$.

In 2019, Jedwab and Li \cite{JL2018} gave the following exponent bound for SEDFs with $m \geq 2$.

\begin{lemma} {\rm\cite{JL2018}}
\label{}
Suppose there exists a $(v, m, k, \lambda)$-SEDF in a group $G$. Let $p$ and $q$ be primes such that $p^d||v$ and $q^f ||\lambda$ for some positive integers $d$ and $f$, and suppose that $q$ is a primitive root modulo $p^d$. Let $G_p$ be the Sylow $p$-subgroup of $G$. Then
$$ \text{exp}(G_p) \leq \frac{v}{q^{\lceil \frac{f}{2} \rceil}}.$$
\end{lemma}

In this paper, we present some nonexistence results for nontrivial SEDFs by using character theory, algebraic number theory and the theory of cyclotomic fields. In Section 2, we give some necessary definitions and notations, and some properties of SEDFs. In Section 3, we use the knowledge on decomposition of prime ideals and integer basis of cyclotomic fields to obtain an exponent bound for SEDFs and apply it to rule out some SEDFs with $m>2$. In Section 4, we use the field descent method, a special homomorphism from a group to its cyclic subgroup and Gauss sums to establish some bounds for prime divisors of $v$ and some congruence relations between $k, m$ and $\lambda$ for $(v,m,k,\lambda)$-SEDFs with $m>2$. We apply them to rule out various SEDFs with $m>2$. In Section 5, we use Schmidt's theorem to get another exponent bound for SEDFs and apply it to rule out several SEDFs with $m>2$.

\section{Preliminaries}

In the next section, we will use some tools of algebraic number theory and the theory of cyclotomic fields to prove the nonexistence results of SEDFs. Now we introduce some definitions and notations in group rings.

Let $\mathbb{Z}$ denote the ring of integers. The group ring $\mathbb{Z}[G]$ is defined to be the ring of formal sums
$$\mathbb{Z}[G]=\left \{ \sum\limits_{g\in G} a_gg\colon a_g\in \mathbb{Z}\right \},$$
where the addition is given by
$$\sum\limits_{g\in G}a_gg+\sum\limits_{g\in G}b_gg=\sum\limits_{g\in G}(a_g+b_g)g,$$
and the multiplication is defined by
$$\left (\sum\limits_{g\in G}a_gg\right )\left (\sum\limits_{g\in G}b_gg\right )=\sum\limits_{h\in G}\left (\sum\limits_{g\in G}a_gb_{hg^{-1}} \right )h.$$
If $S$ is a multiset of $G$ where for $g\in G$, $g$ occurs in $S$ exactly $a_g$ times, we will identify $S$ with the group ring element $S=\sum_{g\in G}a_gg$.  
Let $t$ be an integer. For $X=\sum_{g\in G}a_gg\in \mathbb{Z}[G]$, we write $X^{(t)}=\sum_{g\in G}a_gg^t$.

For a finite abelian group $G$, there are exactly $n = |G|$ distinct homomorphisms (called {\em characters}) from $G$ to the multiplicative group of complex numbers. In particular, the character $\chi_0\colon G\rightarrow \mathbb{C}$ defined by $\chi_0(g)=1$ for all $g\in G$ is called the {\bf principal character}. When $G$ is abelian, the product of characters $(\chi\psi)(g) =\chi(g)\psi(g)$ is again a character and provided $G$ is finite, this gives us a group $\widehat{G}$ of characters isomorphic to $G$. So we can label the $|G|$ distinct characters as $\chi_g$, where $g$ runs through all elements of $G$.

Each $\chi$ extends to an algebra homomorphism from the group algebra $\mathbb{Z}[G]$ to $\mathbb{C}$ by
$$\chi\left(\sum_{g\in G}a_gg\right )=\sum_{g\in G}a_g\chi(g).$$
For a subgroup $H$ of $G$, we write $H^{\perp}=\{\chi\in \widehat{G}:~ \chi(g)=1\ {\rm for\ all}\ g\in H \}$. If $\chi\in H^{\perp}$, we say that $\chi$ is {\bf principal} on $H$. Clearly, $|H^{\perp}|=\frac{|G|}{|H|}$ holds.
We begin with two preparatory lemmas.

\begin{lemma}\label{FIF}{\rm \cite[Lemma 3.5]{BJ1999}}{\rm (Fourier inversion formula)}
Let $G$ be a finite abelian group and $X=\sum_{g\in G}a_gg\in \mathbb{Z}[G]$. Then
$$a_g=\frac{1}{|G|}\sum\limits_{\chi\in \widehat{G}}\chi(Xg^{-1}) $$
holds for all $g\in G$.
\end{lemma}

\begin{lemma}\label{integer r}{\rm \cite[Proposition 6.1.1]{IR1992}}
A rational number $r\in \mathbb{Q}$ is an algebraic integer if and only if $r$ is an integer.
\end{lemma}

Martin and Stinson provided the following property of a nonprincipal character for a SEDF.
\begin{lemma}\label{integer r1}{\rm \cite{MS-arXiv}}
Suppose $\{D_1,...,D_m\}$ is a nontrivial $(v,m, k, \lambda)$-SEDF in a group $G$, and let $D=\bigcup\limits_{i=1}^m D_i$. Then exists a nonprincipal character $\chi\in \widehat{G}$ such that $\chi(D)\neq 0$.
\end{lemma}

Suppose $\{D_1,\ldots, D_m\}$ is a nontrivial $(v,m,k,\lambda)$-SEDF in $G$ with $m\geq 3$, and let $D=\bigcup\limits_{i=1}^m D_i$. Let $\chi$ be a nonprincipal character such that $\chi(D)\neq 0$. Set $\alpha_j$ to be the real number $\frac{|\chi(D_j)|^2}{|\chi(D_j)|^2-\lambda}$. By applying the nonprincipal character $\chi$ to Equation (\ref{SEDF11}), we obtain
 \begin{equation}
\label{Dj1}
\chi(D_j)(\overline{\chi(D)-\chi(D_j)})=-\lambda
\end{equation}
 for each $j$. Then conjugate Equation (\ref{Dj1}), multiply both sides by $\chi(D_j)$, and rearrange to give
 \begin{equation}
\label{Dj2}
\chi(D_j)=\alpha_j\chi(D).
\end{equation}
 Substitute for $\chi(D_j)$ from Equation (\ref{Dj2}) into Equation (\ref{Dj1}) to obtain a quadratic equation in $\alpha_j$:
 $$\alpha_j^2-\alpha_j-\frac{\lambda}{|\chi(D)|^2}=0.$$
 The solutions of this equation are
  \begin{equation}
\label{E1}
 \alpha=\frac{1}{2}\left(1+\sqrt{1+\frac{4\lambda}{|\chi(D)|^2}}\ \right ), \ \ \beta=\frac{1}{2}\left (1-\sqrt{1+\frac{4\lambda}{|\chi(D)|^2}}\ \right ).
 \end{equation}
Let $x$ and $y$ be the number of times $\alpha_j$ takes the value $\alpha$ and $\beta$, respectively, as $j$ ranges over $1\leq j \leq m$. Using $\chi(D)=\sum\limits_{i=1}^m\chi(D_i)$, we find from Equation (\ref{Dj2}) that
  $$x\alpha+y\beta=1.$$
  Combine with the counting condition $x+y=m$ to determine $x$ and $y$ as
  \begin{equation}
   \label{E2}
   x=\frac{m}{2}-\frac{m-2}{2\sqrt{1+\frac{4\lambda}{|\chi(D)|^2}}},\ \  y=\frac{m}{2}+\frac{m-2}{2\sqrt{1+\frac{4\lambda}{|\chi(D)|^2}}}.
   \end{equation}
 From Equation (\ref{E2}), in order that $x,y$ are positive integers, it holds that
\begin{equation}
\label{AX}
\sqrt{1+\frac{4\lambda}{|\chi(D)|^2}}=\frac{b_{\chi}}{a_{\chi}}\ \ {\rm for\ each} \ \chi \in \widehat{G},\ \ {\rm where }\ a_{\chi},\ b_{\chi} \in \mathbb{Z},\ \   b_{\chi}>a_{\chi}>0\ {\rm and}\ \  \gcd(b_{\chi},\ a_{\chi})=1.
\end{equation}
Set
\[
\begin{array}{l}
\vspace{0.1cm}\widehat{G}^0=\{ {\rm nonprincipal}\ \chi \in \widehat{G}:~\chi(D)=0 \},\\
\vspace{0.1cm}\widehat{G}^{+}=\{ {\rm nonprincipal}\ \chi \in \widehat{G}:~ \chi(D)\neq 0,\ \chi(D_1)=\frac{a_{\chi}+b_{\chi}}{2a_{\chi}}\chi(D) \}, \\
\vspace{0.1cm}\widehat{G}^{-}=\{ {\rm nonprincipal}\ \chi \in \widehat{G}:~ \chi(D)\neq 0,\ \chi(D_1)=\frac{a_{\chi}-b_{\chi}}{2a_{\chi}}\chi(D) \}.
\end{array}
\]
Then $\widehat{G}=\{\chi_0\}\cup \widehat{G}^0\cup \widehat{G}^{+} \cup \widehat{G}^{-}$, and for $\chi\in \widehat{G}$, $|\chi(D)|^2$ and $|\chi(D_1)|^2$ are listed in Table I.

\medskip
\vspace{3.5cm}
\centerline{{\bf Table I}}
\[
\begin{array}{|l|l|l|}
\hline
\chi\in \widehat{G} & |\chi(D)|^2 &|\chi(D_1)|^2  \\  \hline

\chi=\chi_0 & k^2m^2 & k^2  \\ \hline

\chi\in \widehat{G}^0 & 0 & \lambda \\ \hline

\chi\in \widehat{G}^+ & \frac{4a_{\chi}^2\lambda}{b_{\chi}^2-a_{\chi}^2} & \frac{(b_{\chi}+a_{\chi})\lambda}{b_{\chi}-a_{\chi}} \\ \hline

\chi\in \widehat{G}^- & \frac{4a_{\chi}^2\lambda}{b_{\chi}^2-a_{\chi}^2} & \frac{(b_{\chi}-a_{\chi})\lambda}{b_{\chi}+a_{\chi}} \\ \hline
\end{array}
\]



Jedwab and Li \cite{JL2018} gave the following necessary conditions for the existence of SEDFs with $m>2$.

\begin{theorem} {\rm\cite{JL2018}}
\label{LJ20191}
Let $a_{\chi}, b_{\chi}$ be as defined in Equation {\rm(\ref{AX})} {\rm(}a nontrivial $(v,m,k,\lambda)$-SEDF in $G$ with $m>2${\rm)}. Then

{\rm (1)} $2b_{\chi}\mid \big( b_{\chi}m-a_{\chi}(m-2)\big)$, and $b_{\chi}\mid (m-2)$;

{\rm (2)} $(b_{\chi}-a_{\chi})\mid (b_{\chi}+a_{\chi})\lambda$, and $(b_{\chi}+a_{\chi})\mid (b_{\chi}-a_{\chi})\lambda$;

{\rm (3)} $(b_{\chi}^2-a_{\chi}^2)\mid 4\lambda$, and if $(b_{\chi}+a_{\chi})$ is odd then $(b_{\chi}^2-a_{\chi}^2)\mid  \lambda$.

\end{theorem}

We need the following lemma for a nontrivial $(v,m,k,\lambda)$-SEDF with $m\geq 3$.
\begin{lemma}\label{integer1}
Suppose $\{D_1,\ldots, D_m\}$ is a nontrivial $(v,m,k,\lambda)$-SEDF in $G$ with $m\geq 3$, and set $D=\bigcup\limits_{i=1}^m D_i$. Let $\chi$ be a nonprincipal character $\chi\in \widehat{G}$ such that $\chi(D)\neq 0$. Then $|\chi(D)|^2, |\chi(D_j)|^2$ and $|\chi(D-D_j)|^2$ are positive integers for each $j,1\leq j \leq m$. 
\end{lemma}

\begin{proof} Let $x$ be the number of times $\alpha_j$ takes the value $\alpha$. By Equation (\ref{E2}), $x$ is a positive integer. Since $m\geq 3$, we have $\sqrt{1+\frac{4\lambda}{|\chi(D)|^2}} \in \mathbb{Q}$. 
Then it holds that $ |\chi(D)|^2 \in \mathbb{Q}$. Clearly, $|\chi(D)|^2$ is an algebraic integer. By Lemma \ref{integer r}, we get $|\chi(D)|^2\in \mathbb{Z}$. Since $\sqrt{1+\frac{4\lambda}{|\chi(D)|^2}}\in \mathbb{Q},$ by Equation (\ref{E1}) we have $\alpha_j\in \mathbb{Q}$. Combing with Equation (\ref{Dj2}), it holds that
\[
\begin{array}{l}
\vspace{0.2cm}|\chi(D_j)|^2=\chi(D_j)\overline{\chi(D_j)}=\alpha_j\chi(D)\overline{\alpha_j\chi(D)}=\alpha_j^2|\chi(D)|^2\in \mathbb{Q},\ \ \ {\rm and}\\
|\chi(D\setminus D_j)|^2=\chi(D\setminus D_j)\overline{\chi(D\setminus D_j)}=(1-\alpha_j)\chi(D)\overline{(1-\alpha_j)\chi(D)}=(1-\alpha_j)^2|\chi(D)|^2\in \mathbb{Q}.
\end{array}
\]
Since $|\chi(D_j)|^2$ and $|\chi(D\setminus D_j)|^2$ are algebraic integers, by Lemma \ref{integer r} we also have $|\chi(D_j)|^2$ and $|\chi(D-D_j)|^2$ are integers.
 This completes the proof. \qed
\end{proof}


\section{An exponent bound obtained by prime ideals factorization}

In this section, we present an exponent bound on a group $G$ containing a $(v,m,k, \lambda)$-SEDF with $m>2$ based on the knowledge of decomposition of prime ideals and integer basis of cyclotomic fields, and use it to prove nonexistence results for the case $m>2$.

For a positive integer $n$, we use $\zeta_n$ to denote the primitive $n$-th complex root of unity $e^{\frac{2\pi i}{n}}$. We need the definition of integer basis.
\begin{definition}{\rm \cite[Chapter 1]{N1999}}
Let $B$ be a ring and let $A$ be a subring of $B$. A system of elements $w_1,\ldots,w_n \in B$ is called an {\bf integer basis} of $B$ over $A$ if each $b\in B$ can be written uniquely as a linear combination
$$ b=a_1w_1+a_2w_2+\cdots+ a_nw_n$$
with cofficients $a_i\in A$.
\end{definition}


In 1999, B. Schmidt \cite{S1999} obtained the following necessary lemma on the behavior of the coefficients of cyclotomic integers in basis representations.

\begin{lemma}\label{BSC}{\rm \cite{S1999}}
Let $n=\prod\limits_{i=1}^tp_i^{a_i}$ be the prime power decomposition of a positive integer $n$, and let $k$ be any divisor of $n$, say $k=\prod\limits_{i=1}^sp_i^{b_i}$ with $s\leq t$ and $1\leq b_i\leq a_i$ for $i=1,\ldots,s.$ Then
$$B_{n,k}:=\left\{\prod\limits_{i=1}^s\zeta_{p_i^{a_i}}^{r_i}\prod\limits_{i=s+1}^t\zeta_{p_i}^{k_i}\zeta_{p_i^{a_i}}^{l_i}:~ 0\leq r_i\leq p_i^{a_i-b_i}-1,\ 0\leq k_i\leq p_i-2,\ 0\leq l_i\leq p_i^{a_i-1}-1\right\}$$
is an integeral basis of $\mathbb{Q}(\zeta_n)$ over $\mathbb{Q}(\zeta_k).$ Furthermore, the following holds.

{\rm (a)} Assume that an element $X$ of $\mathbb{Z}[\zeta_n]$ has the form
$$X=\sum\limits_{j=0}^{n-1}b_j\zeta_n^j$$
where $b_0,\ldots, b_{n-1}$ are integers with $0\leq b_j\leq C$ for some constant $C$. Then
$$X=\sum\limits_{x\in B_{n,k}}x\sum\limits_{j=0}^{k-1}c_{xj}\zeta_k^j$$
where the $c_{xj}$’s are integers with $|c_{xj}|\leq 2^{t-s-1}C$ if $t>s$ and $0\leq c_{xj}\leq C$ if $t=s$.

{\rm (b)} If the assumption on the coefficients is replaced by $|b_j|\leq C$, then
$$X=\sum\limits_{x\in B_{n,k}}x\sum\limits_{j=0}^{k-1}c_{xj}\zeta_k^j$$
where the $c_{xj}$’s are integers with $|c_{xj}|\leq 2^{t-s}C$.
\end{lemma}


Let $H$ be a subgroup of an abelian group $G$, and let $\rho: G \rightarrow H$ be the canonical epimorphism. Each $\widehat{\chi}\in \widehat{H}$ extends to a character $\chi \in \widehat{G}$ satisfying $\chi(g) = \widehat{\chi}(\rho(g))$ for every $g \in G$. For each subset $D$ of $G$, it holds that
$$\rho(D)=\sum\limits_{h\in H}c_hh, \ {\rm where}\ 0\leq c_h\leq \frac{|G|}{|H|}.$$
Applying the character $\widehat{\chi}$ to the above equation, we obtain
$$\chi(D)=\widehat{\chi}(\rho(D))=\sum\limits_{h\in H}c_h\widehat{\chi}(h).$$
Applying  Lemma \ref{BSC} with $k=1$, we obtain the following corollary.
\begin{corollary}\label{c1}
Let $G$ be an abelian group of order $n$, and let $H$ be a cyclic subgroup of $G$ of order $d$. Let $\rho : G \rightarrow H$ be the canonical epimorphism, and let $D$ be a subset of $G$. Then for each generator $\widehat{\chi}$   of $\widehat{H}$, it holds that $\widehat{\chi}(\rho(D))=\sum\limits_{x\in B_{d,1}}c_xx$,
where each $c_x$ is an integer satisfying $|c_x|\leq 2^{t-1}\frac{n}{d}$ and $t$ is the number of distinct prime divisors of $d$.
\end{corollary}

\begin{definition}{\rm \cite[Chapter 15]{LF2010}}
An ideal of $\mathbb{Z}[\zeta_n]$ that is invariant under complex conjugation is said to be {\bf self-conjugate}.
\end{definition}
With respect to our current problem, it is useful to know when complex conjugation in  $\mathbb{Z}[\zeta_w]$ fixes all the prime ideal divisors of the ideal $(n)$.

\begin{definition}{\rm \cite[Chapter 15]{LF2010}}
 Let $p,w$ be positive integers, where $p$ is prime. Write $w=p^su$ where $s\geq 0$ and $u$ is coprime to $p$. Then $p$ is {\bf self-conjugate} modulo $w$ if there is an integer $r$ such that $p^r\equiv -1\pmod u.$ An arbitary integer $n$ is {\bf self-conjugate} modulo $w$ if all of its prime divisors are self-conjugate modulo $w$.
\end{definition}

\begin{theorem}\label{self}
{\rm \cite[Corollary 1.4.5]{S2002}}
 Let $u,w$ be positive integers. If $u$ is self-conjugate modulo $w$, then every prime ideal of $\mathbb{Z}[\zeta_w]$ dividing ideal $(u)$ is self-conjugate.
\end{theorem}


\begin{theorem}\label{g-1}{\rm \cite[Theorem 13.2]{IR1992}}
Let $p$ be a prime and $m$ a positive integer with $p \nmid m$. Let $f$ be the smallest positive integer such that $p^f\equiv 1\pmod m.$ Then in the ring of integers $R\subset \mathbb{Q}(\zeta_m)$, it holds that $$(p)=P_1P_2\cdots P_g$$ where each $P_i$ has degree $f$, $\varphi$ is Euler function and $g=\frac{\varphi(m)}{f}$.
\end{theorem}

\begin{theorem}\label{g-2}{\rm \cite[Proposition 1.10.3]{N1999}}
Let $m=\prod\limits_{i=1}^s p_i^{a_i}$ be the prime factorization of $m$, and for every prime number $p_i$, let $f_i$ be the smallest positive integer such that $p_i^{f_i}\equiv 1\pmod {\frac{m}{p_i^{a_i}}}.$
Then in the ring of integers $R\subset \mathbb{Q}(\zeta_m)$, it holds that
$$(p_i)=(P_1P_2\cdots P_{g_i})^{\varphi(p_i^{a_i})},$$
where $P_1,\ldots,P_{g_i}$ are distinct prime ideals, all of degree $f_i$.
\end{theorem}

Let $F$ be a cyclotomic field generated by $\zeta_m$ over $\mathbb{Q}$ and $G$ the Galois group of $F/\mathbb{Q}$. Then we have the following isomorphism of groups.

\begin{lemma}\label{BOOK1}{\rm \cite[Corollary 13.2.2]{IR1992}}
Let $F=\mathbb{Q}(\zeta_m)$ and $G$ the Galois group of $F/\mathbb{Q}$. Then there is an isomorphism $\theta:G \rightarrow \mathbb{Z}_m^*$ such that for $\sigma\in G$
$$\sigma(\zeta_m)=\zeta_m^{\theta(\sigma)},$$
where $\mathbb{Z}_m^*=\{1\leq a\leq m:~(a,m)=1\}.$
\end{lemma}

With the automorphism Gal$(F/\mathbb{Q})\cong \mathbb{Z}_{m}^*: \sigma_a\mapsto a\pmod {m}$. Let $P$ be a prime ideal of the ring of integers in $F$ and $G$ the Galois group of $F/\mathbb{Q}$. Define $G(P)=\{\sigma\in G:~ \sigma(P)=P\}$. Then we have the following two lemmas.

\begin{lemma}\label{BOOK2}{\rm \cite[Chapter 13]{IR1992}}
Let $p$ be a prime and $m$ a positive integer with $p\nmid m$. Let $(p)=P_1P_2\cdots P_g$ in $\mathbb{Q}(\zeta_m)$ and $P$ be one of the $P_i$. Then $G(P)$ is a cyclic group generated by $\sigma_{p}$.
\end{lemma}

\begin{remark}
If there exists an integer $j$ such that $p^j\equiv -1\pmod m$, by Lemma \ref{BOOK2} it holds that $\sigma_{-1}(P_i)=P_i$, where $(p)=P_1P_2\cdots P_{g}$ and $1\leq i\leq g.$
\end{remark}

\begin{lemma}\label{KSW}{\rm \cite{KSW1985}}
Let $p$ be a prime and $m$ a positive integer with $p\nmid m$. Let $(p)=(P_1P_2\cdots P_{g})^{\varphi(p^{a})}$ in $\mathbb{Q}(\zeta_{p^am})$ and $P$ be one of the $P_i$. If
$$p^l\equiv -1\pmod m$$
for some integer $l$, then $\sigma_{-1}\in G(P)$.
\end{lemma}


The following is a result of Turyn \cite{T1965}.

\begin{lemma}{\rm \cite{T1965}}\label{T19}
Suppose that $A\in \mathbb{Z}[\zeta_m]$ satisfies
$$A\overline{A}\equiv 0\pmod {n^2}$$
for some positive integer $n$ which is self-conjugate modulo $m$. Then $A\equiv 0\pmod n.$
\end{lemma}

Let $u$ be a positive integer and let $p$ be a prime. Let $p^a|| u$ be the integer $a$ defined by the property that $p^a\mid u$ but that $p^{a+1} \nmid u$.

In order to establish the exponent bound in Theorem \ref{SEDF1}, we will need the following generalization of Lemma \ref{T19}.


\begin{lemma}\label{X1}
Let $u=\prod\limits_{i=1}^{t}p_i^{u_i}$ be the prime power decomposition of $u$. Let $w$ be an integer such that $u$ is self-conjugate modulo $w$. Let $p_i\mid w$ for $i=1,\ldots,s$ and $p_i\nmid w$ for $i=s+1,\ldots,t$. Suppose that $X,\overline{X} \in \mathbb{Z}[\zeta_{w}]$ satisfy
$X\overline{X}\equiv 0 \pmod u$. Then $$X \equiv 0\pmod {\prod\limits_{i=1}^{s}p_i^{\lfloor\frac{u_i}{2}\rfloor}\prod\limits_{i=s+1}^{t}p_i^{\lfloor\frac{u_i+1}{2}\rfloor}}.$$
Furthermore, for any $i$, $s+1\leq i\leq t$, if $p_i^{u_i}||u$ and $X\overline{X}=u$, then $u_i$ is even.
\end{lemma}

\begin{proof}
Since $p_i\mid w$ for $1\leq i\leq s$ and $p_i \nmid w$ for $s<i\leq t,$ we can set $w=(\prod\limits_{i=1}^{s}p_i^{a_i})(\prod\limits_{i=1}^{s'}q_i^{b_i})$, where $p_1,\ldots, p_s, q_1,\ldots,q_{s'}$ are distinct primes. Since $p_i \nmid w$ for $s<i\leq t$, by Theorem \ref{g-1} we obtain that each prime divisor $(p_i)$ has the factorization in $\mathbb{Z}[\zeta_{w} ]$
$$(p_i)=(P_{i1}P_{i2}\ldots P_{ig_i}),$$
where $f_i$ is the smallest positive integer such that $p_i^{f_i}\equiv 1\pmod w$ and $g_i=\frac{\varphi(w)}{f_i}$. 
Since $p_i \mid w$ for $1\leq i\leq s$, by Theorem \ref{g-2} we obtain that each prime divisor $(p_i)$ has the factorization in $\mathbb{Z}[\zeta_{w} ]$
$$(p_i)=(P_{i1}P_{i2}\ldots P_{ig_i})^{\varphi(p_i^{a_i})}.$$
 Then, $(u)$ has the factorization in $\mathbb{Z}[\zeta_{w} ]$ 
$$(u)=\prod\limits_{i=1}^s(p_{i})^{u_{i}}\prod\limits_{i=s+1}^t(p_{i})^{u_{i}}=\prod\limits_{i=1}^{s}(P_{i1}P_{i2}\ldots P_{ig_i})^{\varphi(p_i^{a_i})u_i}\prod\limits_{i=s+1}^{t}(P_{i1}P_{i2}\ldots P_{ig_i})^{u_i}.$$
Since $u$ is self-conjugate modulo $w$ and $X\overline{X}\equiv 0 \pmod u$, by Lemmas \ref{BOOK2} and \ref{KSW} we obtain
$$X\equiv0 \pmod{\prod\limits_{i=1}^{s}(P_{i1}P_{i2}\ldots P_{ig_i})^{\left\lfloor\frac{\varphi(p_i^{a_i})u_i+1}{2}\right\rfloor}\prod\limits_{i=s+1}^{t}(P_{i1}P_{i2}\ldots P_{ig_i})^{\left\lfloor\frac{u_i+1}{2}\right\rfloor}}.$$ Hence, we have
$$X\equiv0 \pmod{\prod\limits_{i=1}^{s}p_i^{\left\lfloor\frac{u_i}{2}\right\rfloor}\prod\limits_{i=s+1}^{t}p_i^{\left\lfloor\frac{u_i+1}{2}\right\rfloor}}.$$

Now suppose $p_i^{u_i}||u$ for some $i$, $s+1\leq i\leq t$. Since $u$ is self-conjugate modulo $w$, then we have that $p_i$ is self-conjugate modulo $w$. It holds that
$$(P_{i1}P_{i2}\ldots P_{ig_i})^{\left\lfloor\frac{u_i+1}{2}\right\rfloor} \mid X,\ {\rm and}\ \ \overline{P_{i1}P_{i2}\ldots P_{ig_i}}=P_{i1}P_{i2}\ldots P_{ig_i}.$$
Hence, we have
$$p_i^{2\left\lfloor\frac{u_i+1}{2}\right\rfloor}=(P_{i1}P_{i2}\ldots P_{ig_i})^{2\left\lfloor\frac{u_i+1}{2}\right\rfloor}\mid  X \overline{X}=u.$$
Since  $p_i^{u_i}||u$, we have
$$p_i^{u_i}=(P_{i1}P_{i2}\ldots P_{ig_i})^{u_i}|| X \overline{X}.$$
Therefore, $u_i$ is even. This completes the proof. \qed
\end{proof}

\begin{remark}\label{upi}
Let $u$ be an integer and let $p$ be a prime divisor of $u$ such that $p^c|| u$. Let $w$ be an integer such that $p\nmid w$ and $p$ is self-conjugate modulo $w$. Suppose that $X,\overline{X} \in \mathbb{Z}[\zeta_{w}]$ satisfy $X\overline{X}=u$, from the proof of Lemma \ref{X1}, we know that $c$ is an even integer.
\end{remark}

We are in a position to prove the following exponent bound.
\begin{theorem}\label{SEDF1}
 Suppose there exists a $(v,m, k, \lambda)$-SEDF in $G$ with $m\geq 3$. Let $d>2$ be an integer such that $d\mid \text{exp}(G)$. Let $\lambda'$ be a divisor of $\lambda$ such that  
 $\lambda'$ is self-conjugate modulo $d$ and $\gcd(\lambda',d)=1$.
 Then there exists a divisor $c$ of $\lambda'$ such that
$$ \sqrt{\lambda'}\leq c \leq 2^{t-1}\frac{v}{d},$$
where $t$ is the number of distinct prime divisors of $d$.
\end{theorem}

\begin{proof}
Let $\{D_1,\ldots, D_m\}$ be a $(v,m, k, \lambda)$-SEDF and set $D=\bigcup\limits_{i=1}^m D_i$. Since $d \mid \text{exp}(G)$, there exists a nonprincipal character $\chi$ of $G$ of order $d$. Set $\lambda'=\prod\limits_{i=1}^{s}p_i^{z_i}$. We distinguish two cases.

{\bf Case (i)} $\chi(D)=0$.

By Equation (\ref{Dj1}), we have $\chi(D_{j})\overline{\chi(D_{j})}=\lambda.$
Since $\lambda'\mid \lambda$, we have
$$\chi(D_{j})\overline{\chi(D_{j})}\equiv 0 \pmod{\lambda'}.$$
By Lemma \ref{X1} with gcd$(\lambda', d)=1$, we have $\prod\limits_{i=1}^{s}p_i^{\left\lfloor\frac{z_i+1}{2}\right\rfloor}\mid \chi(D_j).$
Since $\chi(D_j)\in \mathbb{Z}[\zeta_d]$, by Lemma \ref{BSC} we may write
\begin{equation}\label{semi-0}
\chi(D_j)=\left(\prod\limits_{i=1}^{s}p_i^{\left\lfloor\frac{z_i+1}{2}\right\rfloor}\right)\sum\limits_{x\in B_{d,1}}c_xx,
\end{equation}
where each $c_x$ is an integer and $B_{d,k}$ is as defined in Lemma \ref{BSC}. 
By Corollary \ref{c1}, we also have $\chi(D_j)=\sum\limits_{x\in B_{d,1}}d_xx$, where each $d_x$ is an integer satisfying $-2^{(t-1)}\frac{v}{d}\leq d_x \leq 2^{(t-1)}\frac{v}{d}$ and so $\prod\limits_{i=1}^{s}p_i^{\left\lfloor\frac{z_i+1}{2}\right\rfloor}c_x=d_x$ for each $x$. In view of Equation (\ref{Dj1}),  there exists a $d_x\not=0$. Then by Equation (\ref{semi-0}), we obtain
$$\prod\limits_{i=1}^{s}p_i^{\lfloor\frac{z_i+1}{2}\rfloor} \leq 2^{t-1}\frac{v}{d}.$$
Hence, $c=\prod\limits_{i=1}^{s}p_i^{\left\lfloor\frac{z_i+1}{2}\right\rfloor}$ is a divisor of $\lambda'$ satisfying
$$ \sqrt{\lambda'}\leq c \leq 2^{t-1}\frac{v}{d}.$$

{\bf Case (ii)} $\chi(D)\neq0$.

Since $m\geq 3$, by Equations (\ref{E1}) and (\ref{AX}), we have $\chi(D_{l})\overline{\chi(D_{l})}= \frac{b_{\chi}+a_{\chi}}{b_{\chi}-a_{\chi}}\lambda$ and $\chi(D_{l'})\overline{\chi(D_{l'})}$\\$= \frac{b_{\chi}-a_{\chi}}{b_{\chi}+a_{\chi}}\lambda$ for some integers $l$ and $l'$. By Lemma \ref{integer1} with $m\geq 3$, we obtain $\chi(D_{l})\overline{\chi(D_{l})}$ and $\chi(D_{l'})\overline{\chi(D_{l'})}$ are positive integers.

If $\gcd(a_{\chi}+b_{\chi},b_{\chi}-a_{\chi})=1$, we have $\gcd(\frac{b_{\chi}+a_{\chi}}{b_{\chi}-a_{\chi}}\lambda,\frac{b_{\chi}-a_{\chi}}{b_{\chi}+a_{\chi}}\lambda)=\frac{\lambda}{b_{\chi}^2-a_{\chi}^2}$. Then it holds that $\chi(D_{l})\overline{\chi(D_{l})}= (b_{\chi}+a_{\chi})^2\frac{\lambda}{b_{\chi}^2-a_{\chi}^2}$ and $ \chi(D_{l'})\overline{\chi(D_{l'})}= (b_{\chi}-a_{\chi})^2\frac{\lambda}{b_{\chi}^2-a_{\chi}^2}$. For each prime divisor $p_i$ of $\lambda'$, let $w_i$ be the maximum nonnegative integer such that $p_i^{w_i}\mid \frac{\lambda}{b_{\chi}^2-a_{\chi}^2}$. 
Since gcd$(\lambda',d)=1$, by Lemma \ref{X1} we know that $2\mid w_i$ for each $i,1\leq i\leq s.$ Let gcd$(b_{\chi}-a_{\chi},\lambda')=u_2$, gcd$(b_{\chi}+a_{\chi},\lambda')=u_3$ and set $u_1=\prod\limits_{i=1}^{s}p_i^{w_i}$. 
Since $2\mid w_i$, we know that $u_4=\sqrt{u_1}$ is an integer. By Lemma \ref{X1} we have
$$u_3 u_4\mid \chi(D_{l})\ {\rm and}\ u_2 u_4\mid \chi(D_{l'}).$$

Since $\lambda'\mid \lambda$, it holds that $\lambda'\mid u_1u_2u_3.$ Then we have $\lambda'\mid \gcd(\lambda', u_2 u_4)\gcd(\lambda', u_3 u_4).$ Without loss of generality, we assume $\gcd(\lambda', u_2 u_4)\leq \gcd(\lambda', u_3 u_4)$. Since $\chi(D_l)\in \mathbb{Z}[\zeta_d]$, by Lemma \ref{BSC} we may write
$$\chi(D_l)=\big{(}u_3u_4 \big{)}\sum\limits_{x\in B_{d,1}}c_xx,$$
where each $c_x$ is an integer. 
By Corollary \ref{c1}, we also have $\chi(D_l)=\sum\limits_{x\in B_{d,1}}d_xx$, where each $d_x$ is an integer satisfying $-2^{(t-1)}\frac{v}{d}\leq d_x \leq 2^{(t-1)}\frac{v}{d}$ and so $\big{(}u_3 u_4\big{)}c_x=d_x$ for each $x$. In view of Equation (\ref{Dj1}),  there exists a $d_x\not=0$. Then we obtain
$$u_3 u_4 \leq 2^{t-1}\frac{v}{d}.$$

Since gcd$(\lambda', u_2u_4)\leq$ gcd$(\lambda', u_3u_4)$ and $\lambda' \mid u_2u_3u_4^2$, we have gcd$(\lambda', u_3u_4)\geq \sqrt{\lambda'}$. Hence, $c=\gcd(\lambda', u_3 u_4)$ is a divisor of $\lambda'$ satisfying
$$ \sqrt{\lambda'}\leq c \leq 2^{t-1}\frac{v}{d}.$$

If gcd$(a_{\chi}+b_{\chi},b_{\chi}-a_{\chi})=2$, we have gcd$(\frac{b_{\chi}+a_{\chi}}{2},\frac{b_{\chi}-a_{\chi}}{2})=1$ and gcd$(\frac{b_{\chi}+a_{\chi}}{b_{\chi}-a_{\chi}}\lambda,\frac{b_{\chi}-a_{\chi}}{b_{\chi}+a_{\chi}}\lambda)=\frac{4\lambda}{b_{\chi}^2-a_{\chi}^2}$. Then we have $\chi(D_l)\overline{\chi(D_l)}= \big{(}\frac{b_{\chi}+a_{\chi}}{2}\big{)}^2\frac{4\lambda}{b_{\chi}^2-a_{\chi}^2}$ and $ \chi(D_{l'})\overline{\chi(D_{l'})}= \big{(}\frac{b_{\chi}-a_{\chi}}{2}\big{)}^2\frac{4\lambda}{b_{\chi}^2-a_{\chi}^2}$.  Similarly, there exists a divisor $c$ of $\lambda'$ such that
$$ \sqrt{\lambda'}\leq c \leq 2^{t-1}\frac{v}{d}.$$
This completes the proof. \qed
\end{proof}

By using Theorem \ref{SEDF1}, we obtain the following corollary which is also obtained by Jedwab and Li \cite{JL2018}.

\begin{corollary}\label{primitive11}{\rm \cite{JL2018}}
 Suppose there exists a $(v,m, k, \lambda)$-SEDF in $G$ with $m\geq 3$. Let $p$ and $q$ be primes such that
 $p^e|| v$ and $q^f || \lambda$ for some positive integers $e$ and $f$, and suppose that $q$ is a primitive root modulo $ p^e$. Let $G_p$ be the Sylow $p$-subgroup of $G$. Then
$$exp(G_p)\leq \frac{v}{q^ {\lceil f/2 \rceil}}.$$
\end{corollary}

\begin{proof}
 Set $d=\text{exp}(G_p)$, it holds that $d\mid p^e$. Since $q$ is a primitive root modulo $ {p^e}$, there exists a positive integer $u$ such that $q^u\equiv -1 \pmod  {p^e}$ and $q^u\equiv -1 \pmod  {d}$. Thus, $q^f$ is self-conjugate modulo $d$.  By applying Theorem \ref{SEDF1} with $d=exp(G_p)$ and $\lambda'=q^f$, we obtain
 $q^{\lfloor\frac{f+1}{2}\rfloor}\leq \frac{v}{d}.$ Therefore, $\text{exp}(G_p)\leq \frac{v}{q^ {\lceil f/2 \rceil}}.$ \qed
\end{proof}

By directly applying Theorem \ref{SEDF1}, we have obtained the following nonexistence results for SEDFs.

\begin{theorem}
 Let $p_1, p_2$ be two distinct primes. Let $\delta$ be either $0$ or $1$, and let $m, a, b, c$ be four positive integers such that  $m>2$.

{\rm (1)} Let $q_1, \ldots, q_t$ be $t$ distinct primes such that $q_i\mid \left((m-1)p_1^{2a+\delta}+1\right)$ for $i=1,\ldots, t$. If $p_1$ is self-conjugate modulo $q_1\cdots q_t$ and $q_1\cdots q_t>2^{t-1}\left((m-1)p_1^{2a-b}+p_1^{-b}\right)$, then an $((m-1)p_1^{2a+\delta}+1,\\ m,p_1^{a+b+\delta}, p_1^{2b+\delta})$-SEDF does not exist.

{\rm (2)} Let $q_1, \ldots, q_t$ be $t$ distinct primes such that $q_i\mid \left((m-1)p_1^{2a+\delta}p_2^{2c}+1\right)$ for $i=1,\ldots, t$. If $p_1$ is self-conjugate modulo $q_1\cdots q_t$ and $q_1\cdots q_t>2^{t-1}\left((m-1)p_1^{2a-b}p_2^{2c}+p_1^{-b}\right)$, then an $((m-1)p_1^{2a+\delta}p_2^{2c}+1, m, p_1^{a+b+\delta}p_2^c, p_1^{2b+\delta})$-SEDF does not exist.

{\rm (3)} Let $q_1, \ldots, q_t$ be $t$ distinct primes such that $q_i\mid \left((m-1)p_1^{2a}+1\right)$ for $i=1,\ldots, t$. If $p_2$ is self-conjugate modulo $q_1\cdots q_t$ and $q_1\cdots q_tp_2^b>2^{t-1}\left((m-1)p_1^{2a}+1\right)$, then an $((m-1)p_1^{2a}+1, m, p_1^{a}p_2^b,\\ p_2^{2b})$-SEDF does not exist.
\end{theorem}

We illustrate the use of Theorem \ref{SEDF1} to rule out the existence of some other SEDFs in certain abelian groups.

\begin{example}
There does not exist a $(3381,23,130,110)$-SEDF in $\mathbb{Z}_{3381}$.
\end{example}
\begin{proof}
Suppose, for a contradiction, that there exists a $(3381,23,130,110)$-SEDF in $\mathbb{Z}_{3381}$. Since $5^{33}\equiv -1\pmod {3381}$, by applying Theorem \ref{SEDF1} with $\lambda'=5$ and $d=3381$, we have $5\leq 4$, a contradiction. \qed
\end{proof}

\begin{example}
There does not exist a $(4375, 7, 486, 324)$-SEDF in any abelian group $G$ with $5^4\mid \text{exp}(G)$.
\end{example}
\begin{proof}
Suppose, for a contradiction, that there exists a $(4375, 7, 486, 324)$-SEDF in $G$ with $5^4\mid \text{exp}(G)$. Since $3^{250}\equiv -1\pmod {5^4}$ and  $2^{250}\equiv -1\pmod {5^4}$, by  applying Theorem \ref{SEDF1} with $\lambda'=324$ and $d=5^4$, we have $18\leq 7$, a contradiction. \qed
\end{proof}


\section{The field descent method of SEDFs}
In this section, we study the structure of elements $X,Y\in \mathbb{Z}[\zeta_{p^a}]$ which satisfy $XY=N$, where $X=r_YY$ and $r_Y\in \mathbb{Q}$. Our results are based on the tools for the field descent method, a special homomorphism from $G$ to a cyclic group and Gauss sums. Firstly, we give the following useful lemma.

\begin{lemma}\label{FIF0}
Let $D$ be a subset of a finite abelian group $G$. Let $p>2$ be a prime divisor of $|G|$ and write $|G|=p^av$ with $\gcd(v,p)=1$. Suppose $p^a \nmid |D|$, then there exists a nonprincipal character $\chi$ of $G$ of order $p^s$ such that $\chi(D)\neq 0$ for some positive integer $s$ with $1\leq s\leq a$.
\end{lemma}

\begin{proof}
Since $G$ is a finite abelian group and $|G|=p^av$ with gcd$(v,p)=1$, there exist two subgroups $W$ and $H$ such that $G\cong W\times H$, where gcd$(|H|,p)=1$ and $|W|=p^a$. For each character $\chi\in H^{\perp}$, $\chi$ induces the character $\chi|_W$ of $W$. Then we have $\widehat{W}\cong \{\chi|_W:~ \chi\in H^{\perp}\}$.

Set $\text{exp}(W)=p^t.$ Now suppose $\chi(D)=0$ for all nonprincipal character $\chi$ of $G$ of order $p^s$, where $s$ is a positive integer with $1\leq s\leq t$. Let $\sigma: G=W\times H \rightarrow W$ be the natural projection, then we have $\sigma(D)\in \mathbb{Z}[W]$. Set $\sigma(D)=\sum_{g\in W}a_gg\in \mathbb{Z}[W]$. For any character $\chi\in H^{\perp}$, since gcd$(|H|,p)=1$, we have that the order of $\chi$ is $p^s$. Since $\chi(D)=0$ for each nonprincipal character $\chi\in H^{\perp}$, we have $\chi|_W\big{(}\sigma(D)\big{)}=\chi(D)=0$ for each nonprincipal character $\chi|_W \in \widehat{W}$. Since $|\sigma(D)|=|D|$, by Lemma \ref{FIF} we have
$$a_g=\frac{1}{|W|}\sum\limits_{\chi\in \widehat{W}}\chi(\sigma(D)g^{-1})=\frac{1}{p^a}|D|, \ \ {\rm for\ all\ }g\in W, $$
contradicting the assumption that $p^a \nmid |D|$. This completes the proof. \qed
\end{proof}

We derive further divisibility conditions on the parameters of SEDFs in Theorem \ref{k(m-2)1}. We first require the following useful lemmas.

\begin{lemma}{\rm \cite{LS2019}}\label{KN}
Let $H$ be a cyclic group of order $p^a$ with a generator $g$ and let $P$ be its subgroup of order $p$. Then the map $\mathbb{Z}[H]\rightarrow \mathbb{Z}[\zeta_{p^a}]$, $\sum_{0\leq i<p^a} a_ig^i\rightarrow \sum_{0\leq i<p^a} a_i \zeta_{p^a}^i$ is a ring homomorphism with kernal $\{PY:~Y\in \mathbb{Z}[H]\}$.
\end{lemma}

\begin{lemma}\label{FIF1}
Let $p$ be a prime and let $G$ be a finite abelian group with $\text{exp}(G)=p^a$ for some positive integer $a$. Let $g$ be an element of $G$ with ord$(g)=p^a$ and let $\psi$ be a character of the subgroup $\langle g\rangle$ such that $\psi(g)=\zeta_{p^a}$. Then for each character $\varphi$ of $G$, there exists a homomorphism $\theta$ from $G$ to $\langle g\rangle$ such that $\psi\big{(}\theta(D)\big{)}= \varphi(D)$ for any $D\in \mathbb{Z}[G]$.
\end{lemma}

\begin{proof}
Since $G$ is a finite abelian group, there exist $s$ elements $g_1,\ldots, g_s$ such that $G= \langle g_1\rangle\times \ldots \times \langle g_s\rangle$ where ord$(g_i)=p^{a_i}$ for $1\leq i\leq s$ and $1\leq a_i\leq a$. Since $\text{exp}(G)=p^a$, without loss of generality, we assume that $a_1=a$. 
Let $\varphi$ be a character of $G$ and set $$\varphi(g_i)=\zeta_{p^{a_i}}^{b_i}\ {\rm for\ some}\ b_i.$$
Then we have
$$\varphi(g_1^{i_1}\ldots g_s^{i_s})=\varphi(g_1)^{i_1}\ldots \varphi(g_s)^{i_s}=\zeta_{p^{a_1}}^{b_1i_1}\ldots \zeta_{p^{a_s}}^{b_si_s}\ {\rm for\ each\ }\ g_1^{i_1}\ldots g_s^{i_s} \in G.$$
Define a map $\theta$ from $G$ to $\langle g\rangle$ given by
$$ \theta:~g_1^{i_1}\ldots g_s^{i_s}\rightarrow g^{\sum\limits_{j=1}^sp^{(a-a_j)}b_ji_j}\ {\rm for\ each\ }\ g_1^{i_1}\ldots g_s^{i_s} \in G. $$
Clearly, $\theta:~ G\rightarrow \langle g\rangle$ is a homomorphism. Since $\psi$ is a character of $\langle g\rangle$ such that $\psi(g^i)=\zeta_{p^a}^i$ for each $i$, $0\leq i\leq p^a-1$, it holds that
$$\psi\big{(}\theta(g_1^{i_1}\ldots g_s^{i_s})\big{)}=\psi\left(g^{\sum\limits_{j=1}^sp^{(a-a_j)}b_ji_j}\right)=\zeta_{p^a}^{\sum\limits_{j=1}^sp^{(a-a_j)}b_ji_j}=
\varphi\left(g_1^{i_1}\ldots g_s^{i_s}\right)\ {\rm for\ each\ }\ g_1^{i_1}\ldots g_s^{i_s} \in G. $$
Hence, it holds that $\psi\big{(}\theta(D)\big{)}= \varphi(D)$ for any $D\in \mathbb{Z}[G]$. This completes the proof. \qed
\end{proof}

\begin{theorem}\label{k(m-2)1}
Suppose there exists a $(v,m, k, \lambda)$-SEDF with $m\geq 3$ in an abelian group $G$. Let $p$ be an odd prime
with $p^a || v$ and $q$ a prime divisor of $\lambda$. Then
at least one of the following conditions satisfies:
\begin{itemize}
\item [{\rm (1)}] $k(m-2)\equiv 0\pmod {p}$; 
\item [{\rm (2)}] $p^{a}\mid mk$, and $u\equiv 0\pmod 2$ if $q$ is self-conjugate modulo $p$ and $q^u||\lambda$. 
\end{itemize}

\end{theorem}

\begin{proof}
Let $\{D_1,\ldots, D_m\}$ be a $(v,m, k, \lambda)$-SEDF in $G$ and set $D=\bigcup\limits_{i=1}^mD_i$. Let $\chi$ be a nonprincipal character of $G$. The nonprincipal character $\chi$ falls into one of the following two cases:

{\bf Case (i)} $\chi(D)\neq 0$ for some nonprincipal character $\chi$ of $G$ of order $p^s$ with $1\leq s\leq a$. 

By Equation (\ref{AX}), without loss of generality, we assume that
$$\{\chi(D_j):~ 1\leq j\leq x\}=\left\{ \frac{a_{\chi}+b_{\chi}}{2a_{\chi}}\chi(D)\right\}\ \  {\rm and}\ \  \{\chi(D_j):~ x< j\leq m\}=\left\{  \frac{a_{\chi}-b_{\chi}}{2a_{\chi}}\chi(D)\right\},$$
where $(a_{\chi},b_{\chi})$ is as defined in Equation {\rm (\ref{AX})} and $x$ is the number of times $\chi(D_j)$ takes the value $\frac{a_{\chi}+b_{\chi}}{2a_{\chi}}\chi(D)$.
Set $E=\bigcup\limits_{i=2}^{m-1}D_i$, then it holds that
$$\chi\big{(}E\big{)}=\chi\big{(}D\big{)}-\frac{a_{\chi}+b_{\chi}}{2a_{\chi}}\chi(D)-\frac{a_{\chi}-b_{\chi}}{2a_{\chi}}\chi(D)=0.$$

Since $G$ is a finite abelian group and $|G|=p^av'$ with gcd$(v',p)=1$, there exist two subgroups $W$ and $H$ such that $G\cong W\times H$, where gcd$(|H|,p)=1$ and $|W|=p^a$. Since $\chi$ is a nonprincipal character of $G$ of order $p^s$, then we have $\chi\in H^{\perp}$. Let $\sigma: G=W\times H \rightarrow W$ be the natural projection. Then we have
 $$\chi|_W\big{(}\sigma(D)\big{)}=\chi(D)\neq 0\ {\rm and}\ \chi|_W\big{(}\sigma(E)\big{)}=\chi(E)=0,$$
where $\chi$ induces the character $\chi|_W$ of $W$. Since $W$ is an abelian group of order $p^a$, there exist an element $g\in W$ of order $p^b$ and a subgroup $W_1$ such that $W=\langle g\rangle\times W_1$, where $\text{exp}(W)=p^b$. For the character $\chi|_W$ of $W$, by Lemma \ref{FIF1} there exists a homomorphism $\theta$ from $W$ to $\langle g\rangle$ such that
$$\psi\big{(}\theta(\sigma(D))\big{)}=\chi|_W(\sigma(D))\neq 0\ {\rm and}\ \psi\big{(}\theta(\sigma(E))\big{)}=\chi|_W(\sigma(E))=0,$$
where $\psi(g^i)=\zeta_{p^b}^i$ for each $i$, $0\leq i\leq p^b-1$.

Since $\sigma: G=W\times H \rightarrow W$ is the natural projection and $\theta: W\rightarrow \langle g\rangle$ is a homomorphism, we have
$\theta\sigma: G\rightarrow \langle g\rangle$ is a homomorphism. Then we have $\theta\sigma\big{(}E\big{)}\in \mathbb{Z}[\langle g\rangle].$ Set
$\theta\sigma\big{(}E\big{)}=\sum\limits_{i=0}^{p^b-1}a_ig^i, $
then we have $$\sum\limits_{i=0}^{p^b-1}a_i=|E|.$$
Let $P=\langle g^{p^{b-1}}\rangle$. Then $P$ is the subgroup of $\langle g \rangle$ of order $p$. Since $\psi\big{(}\theta\sigma(E)\big{)}=0$, by Lemma \ref{KN} we have $\theta\sigma(E)=PY$ for some $Y\in \mathbb{Z}[\langle g\rangle]$. It follows that
$$|E|=\sum\limits_{i=0}^{p^b-1}a_i \equiv 0 \pmod p.$$
Thus, we conclude $(m-2)k \equiv 0 \pmod p.$

{\bf Case (ii)} $\chi(D)= 0$ for all nonprincipal character $\chi$ of $G$ of order $p^s$ with $1\leq s\leq a$.

By Lemma \ref{FIF0}, we have $p^a \mid mk$. Let $\phi$ be a nonprincipal character of $G$ of order $p$. By Equation (\ref{Dj1}), we have
$$\phi(D_i)\overline{\phi(D_i)}=\lambda$$
for each $i, 1\leq i\leq m.$ Since $q$ is a prime divisor of $\lambda$ such that $q$ is self-conjugate modulo $p$ and $q^u||\lambda$, applying Lemma \ref{X1}, we have $u$ is an even integer. This completes the proof. \qed
\end{proof}

\begin{remark}\label{0p}
If there exists a nonprincipal character $\chi$ of $G$ of order $p^s$ such that $\chi(D)\neq 0$, from the proof of Theorem {\rm\ref{k(m-2)1}} we have $p\mid (m-2)k$.
\end{remark}

Applying Theorem \ref{k(m-2)1}, we obtain the following nonexistence results for SEDFs.

\begin{theorem}\label{bcm0}
 Let $m, b, c$ be three positive integers such that  $m>2$. 

{\rm (1)} Let $p>2, q$ be two distinct primes such that $p\mid \left (q^{2b}+1\right)$ and $q^{2c+1}||(m-1)$. If $p\nmid (m-2)k$, then a $\left(q^{2b}+1, m, k, \lambda\right)$-SEDF does not exist.

{\rm (2)} Let $q$ be a prime such that $q^{2c}||(m-1)$ or $q \nmid (m-1)$. Let $p$ be an odd prime such that $p\mid \left(q^{2b+1}+1\right)$. If $p\nmid (m-2)k$, then a $\left(q^{2b+1}+1, m, k, \lambda \right)$-SEDF does not exist.
\end{theorem}

\begin{proof}
(1) Suppose, for a contradiction, that $\{D_1,\ldots, D_m\}$ is a $\left(q^{2b}+1, m, k, \lambda\right)$-SEDF in a group $G$ of order $q^{2b}+1$. By Equation (\ref{mk2=lambda(n-1)}), we have $\lambda q^{2b}=(m-1)k^2$. Since $q$ is a prime and $q^{2c+1}||(m-1)$, we have $q^{2d+1}||\lambda$ for some nonnegative integer $d.$ Since $p$ is a prime such that $p\mid \left(q^{2b}+1\right)$, we know that $q$ is self-conjugate modulo $p$. Since $q^{2d+1}||\lambda$, applying Theorem \ref{k(m-2)1}, we deduce that $p\mid (m-2)k$, which leads to a contradiction. 

(2) Suppose, for a contradiction, that $\{D_1,\ldots, D_m\}$ is a $\left(q^{2b+1}+1, m, k, \lambda\right)$-SEDF in a group $G$ of order $q^{2b+1}+1$. By Equation (\ref{mk2=lambda(n-1)}), we have $\lambda q^{2b+1}=(m-1)k^2$. Since $q$ is a prime such that $q^{2c}||(m-1)$ or $q\nmid (m-1)$, we have $q^{2d+1}||\lambda$ for some nonnegative integer $d.$ Since $p$ is a prime such that $p\mid \left(q^{2b+1}+1\right)$, we know that $q$ is self-conjugate modulo $p$. Since $q^{2d+1}||\lambda$, applying Theorem \ref{k(m-2)1}, we deduce that $p\mid (m-2)k$, which leads to a contradiction. This completes the proof. \qed
\end{proof}

Similar to the proof of Theorem \ref{bcm0}, applying Theorem \ref{k(m-2)1}, we obtain the following nonexistence results for SEDFs. We omit the proof here. 

\begin{theorem}
 Let $l, m, b, c$ be four positive integers such that $l>2$ and $m>2$.

{\rm (1)} Let $p>2, q$ be two distinct primes such that $p\mid \left(q^{2b}l^2+1\right)$ and $q^{2c+1}||(m-1)$. If $q$ is self-conjugate modulo $p$ and $p\nmid (m-2)k$, then a $\left(q^{2b}l^2+1, m, k, \lambda\right)$-SEDF does not exist.

{\rm (2)} Let $q$ be a prime such that $q^{2c}||(m-1)$ or $q \nmid (m-1)$. Let $p$ be an odd prime such that $p\mid \left(q^{2b+1}l^2+1\right)$. If $q$ is self-conjugate modulo $p$ and $p\nmid (m-2)k$, then a $\left(q^{2b+1}l^2+1, m, k, \lambda\right)$-SEDF does not exist.
\end{theorem}

Applying Theorem \ref{k(m-2)1}, we obtain the following corollary.
\begin{corollary}\label{k(m-2)}
Suppose there exists a $(v,m, k, \lambda)$-SEDF in $G$ with $m\geq 3$. Let $p>2$ be a prime divisor of $v$ and write $v=p^av'$ with $\gcd(v',p)=1$.  If $p^a \nmid mk$, then $k(m-2)\equiv 0\pmod {p}$ holds.
\end{corollary}

We now illustrate the use of Corollary \ref{k(m-2)} to rule out the existence of a $(2500,35,42,24)$-SEDF and a $(6400,80,54,36)$-SEDF.
\begin{example}
There does not exist a $(2500,35,42,24)$-SEDF in any abelian group $G$.
\end{example}
\begin{proof}
Suppose, for a contradiction, that there exists a $(2500,35,42,24)$-SEDF in $G$. Since $5^4\mid 2500$ and $5^4\nmid 35\times 42$, by applying Corollary \ref{k(m-2)} with $p=5$ and $a=4$, we obtain $5\mid 33\times 42$, a contradiction. \qed
\end{proof}

\begin{example}
There does not exist a $(6400,80,54,36)$-SEDF in any abelian group $G$.
\end{example}
\begin{proof}
Suppose, for a contradiction, that there exists a $(6400,80,54,36)$-SEDF in $G$. Since $5^2\mid 6400$ and $5^2\nmid 80\times 54$, by applying Corollary \ref{k(m-2)} with $p=5$ and $a=2$, we obtain $5\mid 78\times 54$, a contradiction.  \qed
\end{proof}

We derive further divisibility conditions on the SEDF parameters in Theorem \ref{pf}. We first require the following number-theoretic lemmas.

\begin{lemma}{\rm \cite{H2010}}\label{IB}
Let $p$ be a prime and let $b$ be a positive integer. Let $X=\sum_{i=0}^{p^{b}-1}c_i\zeta_{p^b}^i$, where each $c_i$ is an integer. Then $X=0$ if and only if $c_i=c_j$ for all $i$ and $j$ satisfying $i\equiv j \pmod{p^{b-1}}.$
\end{lemma}

\begin{lemma}\label{NBIB}
Let $p$ be a prime and let $b$ be a positive integer. Let $a_1, a_2$ be two positive integers with $\gcd(a_1,a_2)=1$. Suppose that $X,Y\in \mathbb{Z}[\zeta_{p^b}]$ satisfy $X=\frac{a_1}{a_2}Y$.  
Then it holds that $$X\equiv 0 \pmod{a_1}\ {\rm  and}\ \ Y\equiv 0 \pmod{a_2}.$$
\end{lemma}

\begin{proof}
Set $X=\sum\limits_{i=0}^{p^{b}-1}c_i\zeta_{p^b}^i$ and $Y=\sum\limits_{i=0}^{p^{b}-1}d_{i}\zeta_{p^b}^i$,  where $c_i$ and $d_{i}$ are integers. Since  $X=\frac{a_1}{a_2}Y$, we have $a_2X-a_1Y=0.$ Thus, $\sum\limits_{i=0}^{p^{b}-1}(a_2c_i-a_1d_i)\zeta_{p^b}^i=0$. By Lemma \ref{IB}, we have
$$a_2c_i-a_1d_i=a_2c_j-a_1d_j\ {\rm for\ all}\ i\ {\rm and}\ j\ {\rm with}\ i\equiv j \pmod{p^{b-1}}.$$
Then we have
$$a_2(c_i-c_j)=a_1(d_i-d_j)\ {\rm for\ all}\ i\ {\rm and}\ j\ {\rm with}\ i\equiv j \pmod{p^{b-1}}.$$
Since gcd$(a_1,a_2)=1$, we obtain
$$d_i\equiv d_j \pmod{a_2} \ {\rm and}\ c_i\equiv c_j \pmod{a_1}\ {\rm for\ all}\ i\ {\rm and}\ j\ {\rm with}\ i\equiv j \pmod{p^{b-1}}.$$
Thus, it holds that
\[
\begin{array}{l}
\vspace{0.2cm}X=\sum\limits_{i=0}^{p^{b}-1}c_i\zeta_{p^b}^i=\sum\limits_{i=0}^{p^{b-1}-1}\zeta_{p^b}^{i}\sum\limits_{j=0}^{p-1}(c_{i+p^{b-1}j}-c_i)\zeta_{p^b}^{p^{b-1}j}\equiv 0 \pmod{a_1},\ \ \ {\rm and}\\
Y=\sum\limits_{i=0}^{p^{b}-1}d_i\zeta_{p^b}^i=\sum\limits_{i=0}^{p^{b-1}-1}\zeta_{p^b}^{i}\sum\limits_{j=0}^{p-1}(d_{i+p^{b-1}j}-d_i)\zeta_{p^b}^{p^{b-1}j}\equiv 0 \pmod{a_2}.
\end{array}
\]
This completes the proof. \qed
\end{proof}

Applying Lemma \ref{NBIB}, we obtain some information about $\sum\limits_{i=1}^m\chi(D_i)$.
\begin{corollary}\label{pf1}
Suppose there exists a $(v,m, k, \lambda)$-SEDF, $\{D_1,\ldots, D_m\}$ in $G$ with $m\geq 3$. Let $p>2$ be a prime divisor of $v$ and let $\chi$ be a nonprincipal character of $G$ of order $p^a$ such that $\sum\limits_{i=1}^m\chi(D_i)\neq0$. Then it holds that
\[
\begin{array}{l}
\left\{\begin{array}{ll}
\vspace{0.2cm}\sum\limits_{i=1}^m\chi(D_i)=2a_{\chi} X\ {\rm for\ some}\ X\in \mathbb{Z}[\zeta_{p^a}]\ {\rm such\ that}\ X\overline{X}=\frac{\lambda}{b_{\chi}^2-a_{\chi}^2},  & {\rm  if}  \ {\rm gcd}(a_{\chi}+b_{\chi},2a_{\chi})=1,\  \\
\sum\limits_{i=1}^m\chi(D_i)=a_{\chi} X\ {\rm for\ some}\ X\in \mathbb{Z}[\zeta_{p^a}]\ {\rm such\ that}\ X\overline{X}=\frac{4\lambda}{b_{\chi}^2-a_{\chi}^2},  & {\rm \ otherwise, \ }
\end{array}
\right .
\end{array}
\]
where $a_{\chi}, b_{\chi}$ are as defined in Equation {\rm(\ref{AX})}.
\end{corollary}

\begin{proof}
Set $D=\bigcup\limits_{i=1}^m D_i$. 
Since $\chi(D)\neq 0$, by Equation (\ref{AX}) we have
\begin{equation}\label{XXX}
\chi(D)\overline{\chi(D)}=\frac{4a_{\chi}^2\lambda}{b_{\chi}^2-a_{\chi}^2}.
\end{equation}
Since gcd$(a_{\chi},b_{\chi})=1$, we have gcd$(a_{\chi}+b_{\chi},2a_{\chi})$=1 or 2. We distinguish two cases.

{\bf Case (i)} gcd$(a_{\chi}+b_{\chi},2a_{\chi})$=1.

Since $m\geq 3$, we have $|\widehat{G}^{+}|>0$. Without loss of generality, we assume that $\chi(D_1)= \frac{a_{\chi}+b_{\chi}}{2a_{\chi}}\chi(D)$ and $\chi(D_m)= \frac{a_{\chi}-b_{\chi}}{2a_{\chi}}\chi(D)$. By using Lemma \ref{NBIB}, we obtain
$$\chi(D_1) \equiv 0\pmod {a_{\chi}+b_{\chi}},\ \ \chi(D_m) \equiv 0\pmod {a_{\chi}-b_{\chi}}\ {\rm and}\ \chi(D) \equiv 0\pmod {2a_{\chi}}.$$
 Then there exists an $X\in \mathbb{Z}[\zeta_{p^a}]$ such that
\begin{equation}\label{XX}
\chi(D_1)=\big{(}a_{\chi}+b_{\chi}\big{)}X, \ \ \chi(D_m)=\big{(}a_{\chi}-b_{\chi}\big{)}X\ {\rm and}\ \chi(D)=2a_{\chi}X.
\end{equation}
Pluging Equation (\ref{XX}) in Equation (\ref{XXX}), we have
\begin{equation}\label{XX1}
X\overline{X}=\frac{\lambda}{b_{\chi}^2-a_{\chi}^2}.
\end{equation}

{\bf Case (ii)} gcd$(a_{\chi}+b_{\chi},2a_{\chi})$=2.

In this case, we have gcd$(\frac{a_{\chi}+b_{\chi}}{2},a_{\chi})$=1. Since $\chi(D_1)= \frac{a_{\chi}+b_{\chi}}{2a_{\chi}}\chi(D)$ and $\chi(D_m)= \frac{a_{\chi}-b_{\chi}}{2a_{\chi}}\chi(D)$, by Lemma \ref{NBIB} we have
$$\chi(D_1) \equiv 0\pmod {\frac{a_{\chi}+b_{\chi}}{2}},\ \ \chi(D_m) \equiv 0\pmod {\frac{a_{\chi}-b_{\chi}}{2}}\  {\rm and}\ \chi(D) \equiv 0\pmod {a_{\chi}}.$$
Then there exists an $X\in \mathbb{Z}[\zeta_{p^a}]$ such that
\begin{equation}\label{XX'}
\chi(D_1)=\left(\frac{a_{\chi}+b_{\chi}}{2}\right)X,\ \ \chi(D_m)=\left(\frac{a_{\chi}-b_{\chi}}{2}\right) X \  {\rm and}\ \chi(D)=a_{\chi}X.
\end{equation}
Pluging Equation (\ref{XX'}) in Equation (\ref{XXX}), we have
\begin{equation}\label{XX2}
X\overline{X}=\frac{4\lambda}{b_{\chi}^2-a_{\chi}^2}.
\end{equation}
This completes the proof. \qed
\end{proof}

The key to our results on SEDFs is to study solutions of $XY=N$, where $X=r_YY$ and $r_Y\in \mathbb{Q}$. Recently, K. H. Leung and B. Schmidt \cite{LS2016}, \cite{LS2019} gave the following necessary conditions for the existence of the solutions of $X\overline{X}=N$. For relatively prime integers $t$ and $s$, we denote the multiplicative order of $t$ modulo $s$ by $\text{ord}_s(t)$. 

\begin{theorem}{\rm \cite{LS2016, LS2019} }\label{Kv^2}
Let $p$ be an odd prime and let $a,N$ be positive integers with $\gcd(N,p)=1$. Suppose that $X\in \mathbb{Z}[\zeta_{p^a}]$ satisfies $X\overline{X}=N$. Let $q_1,\ldots,q_s$ be distinct prime divisors of $N$ and set
$$f=\gcd\big{(}\text{ord}_p(q_1),\ldots,\text{ord}_p(q_s)\big{)}.$$
Let $K$ be the subfield of $\mathbb{Q}(\zeta_{p^a})$ with $[\mathbb{Q}(\zeta_{p^a}):K]=f$. Then $X\zeta_{p^a}^i\in K$ for some integer $i$.
\end{theorem}

\begin{theorem}{\rm \cite{LS2019}}\label{N^2+N+1}
Let $p$ be a prime. Let $N$ be a square-free integer not divisible by $p$ and let $q_1,\ldots,q_s$ be the distinct prime divisors of $N$. Write $f=\gcd\big{(}\text{ord}_p(q_1),\ldots,\text{ord}_p(q_s)\big{)}$ and let $t$ be an integer with $\text{ord}_{p^a}(t)=f.$ Let $\sigma$ be the automorphism of $Q(\zeta_{p^a})$ determined by $\zeta_{p^a}^{\sigma}=\zeta_{p^a}^t$. Suppose $X\in \mathbb{Z}(\zeta_{p^a})$ satisfies $X\overline{X}=N$, 
\vspace{0.1cm}then  the following holds.\\
\vspace{0.1cm}{\rm (1)} $p\leq N^2+N+1$. \\
\vspace{0.1cm}{\rm (2)} $f$ is odd.\\
\vspace{0.1cm}{\rm (3)} There exists an integer $i$ such that $(X\zeta_{p^a}^i)^{\sigma}=X\zeta_{p^a}^i$.
\end{theorem}

By using Corollary \ref{pf1} and Theorem \ref{N^2+N+1}, we obtain the following necessary conditions for the existence of SEDFs.

\begin{theorem}\label{pf}
Suppose there exists a $(v,m, k, \lambda)$-SEDF, $\{D_1,\ldots, D_m\}$ in $G$ with $m\geq 3$. 
Let $p>2$ be a prime divisor of $v$. Let $\chi$ be a nonprincipal character of $G$ of order $p^a$ 
such that 
$\lambda$ is not divisible by $p$. 
Let $q$ be a prime divisor of $\lambda$ such that $q^b\mid \lambda$ and $q$ is self-conjugate modulo $p^a$.
Suppose $\frac{4\lambda}{ b_{\chi}^2-a_{\chi}^2}$ and $\lambda$ are nonsquare integers, where $a_{\chi}, b_{\chi}$ are as defined in Equation {\rm(\ref{AX})}. Then it holds that \\
\[
\begin{array}{l}
\vspace{0.2cm}
\left\{\begin{array}{ll}
\vspace{0.2cm}\ \ p\leq \left( \frac{\lambda}{q^{2\lfloor \frac{b+1}{2}\rfloor}}\right )^2+\frac{\lambda}{q^{2\lfloor \frac{b+1}{2}\rfloor}}+1,\ \ &  {\rm if}\ \sum\limits_{i=1}^m\chi(D_i)=0,\ \\
\vspace{0.2cm}\ \ p\leq \left(\frac{\lambda}{b_{\chi}^2-a_{\chi}^2}\right)^2+\frac{\lambda}{b_{\chi}^2-a_{\chi}^2}+1, \ \ & {\rm  if}  \ {\rm gcd}(a_{\chi}+b_{\chi},2a_{\chi})=1 \ {\rm and}\ \sum\limits_{i=1}^m\chi(D_i)\not=0, \\
\ \ p\leq \left(\frac{4\lambda}{b_{\chi}^2-a_{\chi}^2}\right)^2+\frac{4\lambda}{b_{\chi}^2-a_{\chi}^2}+1, \ \ & {\rm  if}  \ {\rm gcd}(a_{\chi}+b_{\chi},2a_{\chi})=2 \ {\rm and}\ \sum\limits_{i=1}^m\chi(D_i)\not=0.
\end{array}
\right .
\end{array}
\]
\end{theorem}

\begin{proof}
Set $D=\bigcup\limits_{i=1}^m D_i$. The nonprincipal character $\chi$ falls into one of the following two cases: 

{\bf Case (i)} $\chi(D)=0$.

By Equation (\ref{Dj1}), it holds that
$$\chi(D_i)\overline{\chi(D_i)}=\lambda.$$
Since $q^b\mid \lambda$, we obtain
$$\chi(D_i)\overline{\chi(D_i)}\equiv 0\pmod {q^b}.$$
Since $p\nmid \lambda$ and $q$ is self-conjugate modulo $p^a$, by Lemma \ref{X1} we have
$$\chi(D_i) \equiv 0\pmod {q^{\lfloor \frac{b+1}{2}\rfloor}}.$$
Thus, there exists an $X_i\in \mathbb{Z}[\zeta_{p^a}]$ such that
$$\chi(D_i)=q^{\lfloor \frac{b+1}{2}\rfloor} X_i, \ \ {\rm and}\ X_i\overline{X_i}=\frac{\lambda}{q^{2\lfloor \frac{b+1}{2}\rfloor}}.$$
Since $\lambda$ is a nonsquare integer, we have $\frac{\lambda}{ q^{2\lfloor \frac{b+1}{2}\rfloor} }$ is a nonsquare integer. By Theorem \ref{N^2+N+1}, we obtain $p\leq \big{(}\frac{\lambda}{q^{2\lfloor \frac{b+1}{2}\rfloor}}\big{)}^2+\frac{\lambda}{q^{2\lfloor \frac{b+1}{2}\rfloor}}+1$.%

{\bf Case (ii)} $\chi(D)\neq 0$.

By Corollary \ref{pf1}, it holds that
\[
\begin{array}{l}
\left\{\begin{array}{ll}
\vspace{0.2cm}\chi(D)=2a_{\chi} X\ {\rm for\ some}\ X\in \mathbb{Z}[\zeta_{p^a}]\ {\rm such\ that}\ X\overline{X}=\frac{\lambda}{b_{\chi}^2-a_{\chi}^2}, \ \ & {\rm  if}  \ {\rm gcd}(a_{\chi}+b_{\chi},2a_{\chi})=1,\  \\
\chi(D)=a_{\chi} X\ {\rm for\ some}\ X\in \mathbb{Z}[\zeta_{p^a}]\ {\rm such\ that}\ X\overline{X}=\frac{4\lambda}{b_{\chi}^2-a_{\chi}^2},  \ \ & {\rm  if}  \ {\rm gcd}(a_{\chi}+b_{\chi},2a_{\chi})=2.
\end{array}
\right .
\end{array}
\]
Since $\frac{4\lambda}{ b_{\chi}^2-a_{\chi}^2}$ is a nonsquare integer, by Theorem \ref{N^2+N+1} we obtain
\[
\begin{array}{l}
\vspace{0.2cm}
\left\{\begin{array}{ll}
\vspace{0.2cm}\ \ p\leq \left(\frac{\lambda}{b_{\chi}^2-a_{\chi}^2}\right)^2+\frac{\lambda}{b_{\chi}^2-a_{\chi}^2}+1, \ \ & {\rm  if}  \ {\rm gcd}(a_{\chi}+b_{\chi},2a_{\chi})=1, \\
\ \ p\leq \left(\frac{4\lambda}{b_{\chi}^2-a_{\chi}^2}\right)^2+\frac{4\lambda}{b_{\chi}^2-a_{\chi}^2}+1, \ \ & {\rm  if}  \ {\rm gcd}(a_{\chi}+b_{\chi},2a_{\chi})=2.
\end{array}
\right .
\end{array}
\]
This completes the proof. \qed
\end{proof}

We now illustrate the use of Theorem \ref{pf} to rule out the existence of a $(6976,218,30,28)$-SEDF.

\begin{example}
There does not exist a $(6976,218,30,28)$-SEDF in any abelian group $G$.
\end{example}
\begin{proof}
Suppose, for a contradiction, that $\{D_1,\ldots, D_{218}\}$  is a $(6976,218,30,28)$-SEDF in $G$. Let $\chi$ be a character of $G$ of order $109$ and set $D=\bigcup\limits_{i=1}^{218} D_i$.

If $\chi\big{(}D\big{)}=0$, by Equation (\ref{Dj1}) we have $\chi\big{(}D_i\big{)}\overline{\chi\big{(}D_i\big{)}}=28.$ Since $2^{18}\equiv -1 \pmod{109}$, by Theorem \ref{pf} with $q=2$, $b=2$, $p=109$ and $\lambda=28$, we have
$$p\leq \left(\frac{\lambda}{q^{2\lfloor \frac{b+1}{2}\rfloor}}\right)^2+\frac{\lambda}{q^{2\lfloor \frac{b+1}{2}\rfloor}}+1=57,$$
 a contradiction. Then we have $\chi\big{(}D\big{)}\neq 0$. From the proof of Corollary \ref{k(m-2)}, we have $109\mid 28\times 28$, again a contradiction. \qed
\end{proof}

Let $\chi$ be a nonprincipal character of $G$ of order $p^a$ such that $\sum_{i=1}^m\chi(D_i)\neq0$. Then by using Corollary \ref{pf1}, Theorems \ref{Kv^2} and \ref{N^2+N+1}, we obtain the following necessary conditions for the existence of SEDFs. We omit the proof here.

\begin{theorem}\label{pfK1}
Suppose there exists a $(v, m, k, \lambda)$-SEDF, $\{D_1,\ldots, D_m\}$ in $G$ with $m\geq 3$. Let $p>2$ be a prime divisor of $v$. Let $\chi$ be a nonprincipal character of $G$ of order $p^a$ such that $\sum\limits_{i=1}^m\chi(D_i)\neq0$ and $\frac{4\lambda}{b_{\chi}^2-a_{\chi}^2}$ is not divisible by $p$, where $a_{\chi}, b_{\chi}$ are as defined in Equation {\rm(\ref{AX})}. Let $q_1,\ldots,q_s$ be the distinct prime divisors of $\frac{4\lambda}{b_{\chi}^2-a_{\chi}^2}$ and set
$$f=\gcd\big{(}\text{ord}_p(q_1),\ldots, \text{ord}_p(q_s)\big{)}.$$
Let $K$ be the subfield of $\mathbb{Q}(\zeta_{p^a})$ with $[\mathbb{Q}(\zeta_{p^a}):K]=f$. Then there exists a $j$ such that
$$\chi(D_i)\zeta_{p^a}^j\in K\ {\rm and}\  \chi(D_i)\overline{\chi(D_i)}=\frac{(a_{\chi}+b_{\chi})\lambda}{b_{\chi}-a_{\chi}}\ {\rm or}\ \frac{(b_{\chi}-a_{\chi})\lambda}{b_{\chi}+a_{\chi}}\ {\rm for\ all}\ i,\ 1\leq i\leq m. $$
\end{theorem}

\begin{lemma}{\rm \cite{IR1992}}\label{GS}
Let $p$ be a prime. Let $\eta$ be the quadratic multiplicative character of $\mathbb{Z}_p$ and $\varphi$ an additive character of $\mathbb{Z}_p$. Then the value of the {\bf quadratic Gauss sum} $G(\eta,\varphi)$ is given by
\[
\begin{array}{l}
G(\eta,\varphi)=\sum\limits_{i=1}^{p-1}\eta(i)\varphi(i)=
\left\{\begin{array}{ll}
\vspace{0.1cm} \sqrt{-p} , \ \ & {\rm  if}  \ p\equiv 3\pmod 4,\  \\
\sqrt{p}, \ \ & {\rm if } \ p\equiv 1\pmod 4.
\end{array}
\right .
\end{array}
\]
\end{lemma}

By using quadratic Gauss sums and Theorem \ref{N^2+N+1}, we obtain the following necessary conditions for the existence of SEDFs.


\begin{theorem}\label{pfK}
Let $p$ be an odd prime and let $N$ be a positive integer with $\gcd(N,p)=1$. Let $H=\langle g\rangle$ be a cyclic group of order $p$ and let $\chi$ be a character of $H$ such that $\chi(g)=\zeta_p$. Suppose there exists an element $D=\sum_{i=0}^{p-1}a_ig^i\in \mathbb{Z}[H]$ with $a_i\geq 0$ such that $\chi(D)\overline{\chi(D)}=N$ and $\sum_{i=0}^{p-1}a_i=k$. Let $q_1,\ldots,q_s$ be distinct prime divisors of $N$ and set
$$f=\gcd\left(\text{ord}_p(q_1),\ldots,\text{ord}_p(q_s)\right).$$
If $\frac{p-1}{2}\mid f$, then one of the following holds.\\
\[
\begin{array}{l}
\vspace{0.2cm} {\rm (1)}\ \  k\equiv t \pmod p,\ k\geq t(1-p)\ {\rm and}\ N=t^2\ {\rm for\ some\ integer}\ t. \\
\vspace{0.2cm} {\rm (2)}\
\left\{\begin{array}{ll}
\vspace{0.2cm}  k\equiv a\pmod p \ {\rm and}\ k\geq \max\{a+pb, a-pb, a(1-p)\},  & {\rm  if}  \ p\equiv 3\pmod 4,  N=a^2+pb^2,  \\
\vspace{0.2cm} k\equiv a \pmod p \ {\rm and}\ k\geq \max\left\{a, a(1-p)\right\}, & {\rm if } \ p\equiv 1\pmod 4,   N=a^2, \\
k\equiv 0 \pmod p \ {\rm and}\ k\geq \max\left\{ap, \ -ap\right\}, & {\rm if } \ p\equiv 1\pmod 4,  N=a^2p,
\end{array}
\right .
\end{array}
\]
where $a,b$ are integers.
\end{theorem}

\begin{proof}
Since $D=\sum_{i=0}^{p-1}a_ig^i\in \mathbb{Z}[H]$ satisfies $\chi(D)\overline{\chi(D)}=N$, by Theorems \ref{Kv^2} and \ref{N^2+N+1}, we have
$\chi(D)\zeta_{p}^c\in K$ for some integer $c$. Since $\frac{p-1}{2}\mid f$, we have $f=\frac{p-1}{2}$ or $p-1$. We distinguish two cases.

{\bf Case (1)} $f=p-1$.

Clearly, $K=\mathbb{Q}$. Since $\chi(D)\zeta_{p}^c$ is an algebraic integer, we have $\chi(D)\zeta_{p}^c\in \mathbb{Z}$. Set $t=\chi(D)\zeta_{p}^c$, then it follows that
$$N=\chi(D)\overline{\chi(D)}=t^2\ {\rm and}\ \chi(D)=t\zeta_{p}^{-c}.$$
Then we have
\[
\begin{array}{l}
\chi\big{(}D-tg^{p-c}\big{)}=\chi(D)-t\zeta_p^{-c}=0.
\end{array}
\]
By Lemma \ref{KN}, there exists an element $Y \in \mathbb{Z}[H]$ with nonnegative coefficients such that
\[
\begin{array}{l}
D-tg^{p-c}=\sum\limits_{j=0}^{p-1}a_jg^j-tg^{p-c}= YH.
\end{array}
\]
By Lemma \ref{IB}, we have $a_{p-c}-t=a_j$ for each $j$ with $j\neq p-c$. Since $\sum_{j=0}^{p-1}a_j=k$ and $a_j\geq 0$, we have
$$k=\sum\limits_{j=0}^{p-1}a_j=pa_{p-c}+t(1-p).$$
Hence, we have $p\mid (k-t)$ and $k\geq t(1-p).$

{\bf Case (2)} $f=\frac{p-1}{2}$.

Then  $[K:\mathbb{Q}]=2$ and $K=\mathbb{Q}(\sqrt{d})$ for some integer $d$. Since $K$ is a quadratic number field,  the rings of algebraic integers in $K$ is
\[
\begin{array}{l}
\left\{\begin{array}{ll}
\vspace{0.1cm} \mathbb{Z}+ \mathbb{Z}\sqrt{-p} , \ \ & {\rm  if}  \ p\equiv 3\pmod 4,\  \\
\mathbb{Z}+ \mathbb{Z}\big{(}\frac{-1+\sqrt{p}}{2}\big{)}, \ \ & {\rm if } \ p\equiv 1\pmod 4.
\end{array}
\right .
\end{array}
\]
Since $\chi(D)\zeta_{p}^c$ is an algebraic integer, there exist two integers $a,b$ such that
\[
\begin{array}{l}
\chi(D)\zeta_{p}^c=
\left\{\begin{array}{ll}
\vspace{0.1cm} a+ b\sqrt{-p} , \ \ & {\rm  if}  \ p\equiv 3\pmod 4,\  \\
a+b\frac{-1+\sqrt{p}}{2}, \ \ & {\rm if } \ p\equiv 1\pmod 4.
\end{array}
\right .
\end{array}
\]
Then we have 
\[
\begin{array}{l}
N=\chi(D)\zeta_{p}^c\overline{\chi(D)\zeta_{p}^c}=
\left\{\begin{array}{ll}
\vspace{0.1cm} a^2+ pb^2 ,  & {\rm  if}  \ p\equiv 3\pmod 4,\  \\
\frac{(2a-b)^2+pb^2+2(2a-b)b\sqrt{p}}{4},  & {\rm if } \ p\equiv 1\pmod 4.
\end{array}
\right .
\end{array}
\]
Since $N$ is a positive integer, we have $(2a-b)b=0$ for $p\equiv 1\pmod 4$. Then the algebraic integer $\chi(D)\zeta_{p}^c$ can be expressed as
\begin{equation}\label{QFN}
\chi(D)\zeta_{p}^c=
\left\{\begin{array}{ll}
\vspace{0.1cm} a+ b\sqrt{-p},  & {\rm  if}  \ p\equiv 3\pmod 4, N=a^2+pb^2,  \\
\vspace{0.1cm}a,  & {\rm if } \ p\equiv 1\pmod 4, N=a^2,\\
a\sqrt{p},  & {\rm if } \ p\equiv 1\pmod 4, N=a^2p.
\end{array}
\right .
\end{equation}
By Equation (\ref{QFN}) and Lemma \ref{GS}, we have
{\small
\[
\chi(D)=
\left\{\begin{array}{ll}
\vspace{0.1cm} \big{(}a+ b\sqrt{-p}\big{)}\zeta_{p}^{-c}=\left(a+b\sum\limits_{j=1}^{p-1}\eta(j)\zeta_p^j \right)\zeta_{p}^{-c}=a\zeta_{p}^{-c}+b\sum\limits_{j=0}^{p-1}\zeta_p^{(j^2-c)} , & {\rm  if}  \ p\equiv 3\pmod 4, N=a^2+pb^2,  \\
\vspace{0.1cm}a\zeta_{p}^{-c}, & {\rm if } \ p\equiv 1\pmod 4, N=a^2, \\
a\sqrt{p}\zeta_{p}^{-c}=\left(a\sum\limits_{j=1}^{p-1}\eta(j)\zeta_p^j\right)\zeta_{p}^{-c}
=a\sum\limits_{j=0}^{p-1}\zeta_p^{(j^2-c)}, & {\rm if } \ p\equiv 1\pmod 4, N=a^2p, \\
\end{array}
\right .
\]
}
where $\eta$ is the quadratic multiplicative character of $\mathbb{Z}_p$.
Then we have
{
\[
\begin{array}{l}
\left\{\begin{array}{ll}
\vspace{0.1cm} \chi\left(D-ag^{p-c}-b\sum\limits_{j=0}^{p-1}g^{(j^2-c)}\right)=0, & {\rm  if}  \ p\equiv 3\pmod 4,\ N=a^2+pb^2, \\
\vspace{0.1cm} \chi\left(D-ag^{p-c}\right)=\chi(D)-a\zeta_{p}^{-c}=0, & {\rm if } \ p\equiv 1\pmod 4,\ N=a^2,\ \\
\chi\left(D-a\sum\limits_{j=0}^{p-1}g^{(j^2-c)} \right)=\chi(D)-a\sum\limits_{j=0}^{p-1}\zeta_p^{(j^2-c)} =0, & {\rm if } \ p\equiv 1\pmod 4,\  N=a^2p. 
\end{array}
\right .
\end{array}
\]
}
By Lemma \ref{KN}, there exist three elements $Y_1, Y_2, Y_3 \in \mathbb{Z}[H]$ with nonnegative coefficients such that
{\small
\[
\begin{array}{l}
\left\{\begin{array}{ll}
\vspace{0.1cm} D-ag^{p-c}-b\sum\limits_{j=0}^{p-1}g^{(j^2-c)}= \sum\limits_{j=0}^{p-1}a_jg^j-ag^{p-c}-b\sum\limits_{j=0}^{p-1}g^{(j^2-c)} =Y_1H, & {\rm  if}  \ p\equiv 3\pmod 4,\ N=a^2+pb^2,  \\
\vspace{0.1cm}D-ag^{p-c}=\sum\limits_{j=0}^{p-1}a_jg^j-ag^{p-c}=Y_2H, & {\rm if } \ p\equiv 1\pmod 4,\ N=a^2\\
D-a\sum\limits_{j=0}^{p-1}g^{(j^2-c)}=\sum\limits_{j=0}^{p-1}a_jg^j-a\sum\limits_{j=0}^{p-1}g^{(j^2-c)} =Y_3H, & {\rm if } \ p\equiv 1\pmod 4,\ N=a^2p.\\
\end{array}
\right .
\end{array}
\]}
Similarly, by Lemma \ref{IB} with $\sum_{j=0}^{p-1}a_j=k$ and $a_j\geq 0$, we have
\[
\begin{array}{l}
\left\{\begin{array}{ll}
\vspace{0.1cm} p\mid (k-a) \ {\rm and}\ k\geq \max\{a+pb, a-pb, a(1-p)\}, & {\rm  if}  \ p\equiv 3\pmod 4,\ N=a^2+pb^2,  \\
\vspace{0.1cm} p\mid \left(k-a\right) \ {\rm and}\ k\geq \max\left\{a, a(1-p)\right\}, & {\rm if } \ p\equiv 1\pmod 4,\ N=a^2, \\
 p\mid k \ {\rm and}\ k\geq \max\left\{ap, \ -ap\right\}, & {\rm if } \ p\equiv 1\pmod 4,\  N=a^2p.
\end{array}
\right .
\end{array}
\]
This completes the proof. \qed
\end{proof}

We now illustrate the use of Theorem \ref{pfK} to rule out the existence of a $(1540,77,18,16)$-SEDF, a $(1701,35,30,18)$-SEDF and a $(2376, 11, 190, 152)$-SEDF.

\begin{example}
There does not exist a $(1540,77,18,16)$-SEDF in any abelian group $G$.
\end{example}
\begin{proof}
Suppose, for a contradiction, that $\{D_1,\ldots, D_{77}\}$ is a $(1540,77,18,16)$-SEDF in $G$ and set $D=\bigcup\limits_{i=1}^{77} D_i$. Let $H$ be a subgroup of $G$ of order $11$ and set $H=\langle g\rangle$. Let  $\sigma: G \rightarrow H$ be the natural projection and let $\chi$ be a character of $H$ of order $11$. 

If $\chi\big{(}\sigma(D)\big{)}=0$, by Equation (\ref{Dj1}) we have $\chi\big{(}\sigma(D_i)\big{)}\overline{\chi\big{(}\sigma(D_i)\big{)}}=16.$ Since
$\sigma(D)=\sum\limits_{i=0}^{10}a_ig^i\in \mathbb{Z}[H]$ with $a_i\geq 0$, applying Theorem \ref{pfK} with $p=11$, $N=16$ and $k=18$, we obtain $t=-4$ and $k=18\geq t(1-p)=40$, a contradiction. Then we have $\chi\big{(}\sigma(D)\big{)}\neq 0.$ By Theorem \ref{LJ20191},
$$(a_{\chi}, b_{\chi})\in \{(1,3),\ (3,5) \}. $$

For $(a_{\chi}, b_{\chi})=(1,3)$, by Equation (\ref{AX}) we have $\chi\big{(}\sigma(D_{l})\big{)}\overline{\chi\big{(}\sigma(D_{l})\big{)}}= \frac{(b_{\chi}-a_{\chi})\lambda}{b_{\chi}+a_{\chi}}=8$ for some integer $l$. By applying Theorem \ref{pfK} with $p=11, q_1=2$ and $f=10$, we have $N=8$ is a square integer, a contradiction.

For $(a_{\chi}, b_{\chi})=(3,5)$, by Equation (\ref{AX}) we have $\chi\big{(}\sigma(D_{l})\big{)}\overline{\chi\big{(}\sigma(D_{l})\big{)}}= \frac{(b_{\chi}-a_{\chi})\lambda}{b_{\chi}+a_{\chi}}=4$ for some integer $l$. By applying Theorem \ref{pfK} with $p=11, q_1=2$ and $f=10$, we have $t=\pm 2$ and $11\mid (18\pm 2)$. This again leads to a contradiction. \qed
\end{proof}

\begin{example}
There does not exist a $(1701,35,30,18)$-SEDF in any abelian group $G$.
\end{example}
\begin{proof}
Suppose, for a contradiction, that $\{D_1,\ldots, D_{35}\}$ is a $(1701,35,30,18)$-SEDF in $G$ and set $D=\bigcup\limits_{i=1}^{35} D_i$. Let $H$ be a subgroup of $G$ of order $7$ and set $H=\langle g\rangle$. Let $\sigma: G \rightarrow H$ be the natural projection and let $\chi$ be a character of $H$ of order $7$.

If $\chi\big{(}\sigma(D)\big{)}=0$, by Equation (\ref{Dj1}) we would have $\chi\big{(}\sigma(D_i)\big{)}\overline{\chi\big{(}\sigma(D_i)\big{)}}=18.$ Since
$\sigma(D)=\sum\limits_{i=0}^{6}a_ig^i\in \mathbb{Z}[H]$ with $a_i\geq 0$, applying Theorem \ref{pfK} with $p=7$, $N=18$ and $f=3$, we obtain $(a,b)$ is a $\mathbb{Z}$-solution of the equation $a^2+ 7b^2=18$, a contradiction. Then we have $\chi\big{(}\sigma(D)\big{)}\neq 0.$  The character $\chi$ of $H$ can be extend to a character $\phi$ of $G$ of order $7$ such that $\phi(D)=\chi\big{(}\sigma(D)\big{)}\neq 0.$ From the proof of Theorem {\rm\ref{k(m-2)1}}, we have $p\mid (m-2)k$, that is, $7\mid 33\times 30.$ This again leads to a contradiction. \qed
\end{proof}

\begin{example}
There does not exist a $(2376, 11, 190, 152)$-SEDF in any abelian group $G$.
\end{example}
\begin{proof}
Suppose, for a contradiction, that $\{D_1,\ldots, D_{11}\}$ is a $(2376, 11, 190, 152)$-SEDF in $G$ and set $D=\bigcup\limits_{i=1}^{11} D_i$. Let $H$ be a subgroup of $G$ of order $11$ and set $H=\langle g\rangle$. Let $\sigma: G \rightarrow H$ be the natural projection and let $\chi$ be a character of $H$ of order $11$.

If $\chi\big{(}\sigma(D)\big{)}=0$, by Equation (\ref{Dj1}) we would have $\chi\big{(}\sigma(D_i)\big{)}\overline{\chi\big{(}\sigma(D_i)\big{)}}=152.$ Since
$\sigma(D)=\sum\limits_{i=0}^{10}a_ig^i\in \mathbb{Z}[H]$ with $a_i\geq 0$, applying Theorem \ref{pfK} with $p=11$, $N=152$ and $f=10$, we obtain $(a,b)$ is a $\mathbb{Z}$-solution of the equation $a^2+ 11b^2=152$. We get a contradiction. Then we have $\chi\big{(}\sigma(D)\big{)}\neq 0.$ By Theorem \ref{LJ20191},
$$(a_{\chi}, b_{\chi})\in \{(1,3),\ (7,9) \}. $$

For $(a_{\chi}, b_{\chi})=(1,3)$, by Equation (\ref{AX}) we have $\chi\big{(}\sigma(D_{l})\big{)}\overline{\chi\big{(}\sigma(D_{l})\big{)}}= \frac{(b_{\chi}-a_{\chi})\lambda}{b_{\chi}+a_{\chi}}=76$. By applying Theorem \ref{pfK} with $p=11$ and $f=10$, we obtain $(a,b)$ is a $\mathbb{Z}$-solution of the equation $a^2+ 11b^2=76$. This leads to a contradiction.

For $(a_{\chi}, b_{\chi})=(7,9)$, by Equation (\ref{AX}) we have $\chi\big{(}\sigma(D_{l})\big{)}\overline{\chi\big{(}\sigma(D_{l})\big{)}}= \frac{(b_{\chi}-a_{\chi})\lambda}{b_{\chi}+a_{\chi}}=19$. By applying Theorem \ref{pfK} with $p=11$ and $f=10$, we obtain $(a,b)$ is a $\mathbb{Z}$-solution of the equation $a^2+ 11b^2=19$. This again leads to a contradiction. \qed
\end{proof}

\begin{remark}
Theorems {\rm\ref{k(m-2)1}}, {\rm\ref{pf}} and {\rm\ref{pfK}} rule out parameter sets of $(v,m,k,\lambda)$-SEDF in any abelian group $G$ which are not excluded by Jedwab and Li {\rm\cite{JL2018}}, including
\[
\begin{array}{l}
 \vspace{0.2cm}(v,m,k,\lambda)\in \{(1540, 77, 18, 16), (1701, 35, 30, 18), (2376, 11, 190, 152), (2500, 35, 42, 24),\\
 \vspace{0.2cm}  \hspace{2.2cm}(2784, 116, 22, 20), (3381, 23, 130, 110), (4564, 163, 26, 24), (4625, 37, 68, 36), \\
  \vspace{0.2cm} \hspace{2.2cm} (5888, 92, 58, 52), (6400, 80, 54, 36), (6976, 218, 30, 28), (8625, 23, 140, 50), \\
\hspace{2.2cm} (8625, 23, 280, 200), (8960, 7, 1054, 744), (9801, 101, 70, 50)  \}.\\
\end{array}
\]

\end{remark}


\section{An exponent bound obtained by Schmidt's theorem}
In this section, we use Schmidt's theorem to get an exponent bound. We now come to an important method due to Schmidt \cite{S1999}, which concerns the solutions $\gamma$ in $\mathbb{Z}[\zeta_M]$ of the norm equation $\gamma \cdot \overline{\gamma}=N$ for a positive integer $N$. Schmidt proved that if $M$ has square divisors, then very often all solutions are of the form $\zeta^i_M \beta$, where $\beta\in \mathbb{Z}[\zeta_{F(M,N)}]$, $F(M,N)$ has the same prime divisors as $M$. Firstly, we introduce the following notation.
\begin{definition}{\rm \cite{S1999}}
Let $M,N$ be positive integers, and let $M=\prod\limits_{i=1}^{t}p_i^{c_i}$ be the prime power decomposition of $M$. For each prime divisor $q$ of $N$, let
\[
\begin{array}{l}
m_q=
\left\{\begin{array}{ll}
\prod\limits_{p_i\not=q}p_i,  & {\rm  if}  \ M {\rm \ is\ odd\ or\ } q=2,\  \\
4\prod\limits_{p_i\not=2, q}p_i,  &  {\rm \ otherwise. \ }
\end{array}
\right .
\end{array}
\]
Let $D(N)$ be the set of prime divisors of $N$. Next define the $F(M,N)=\prod\limits_{i=1}^tp_i^{b_i}$ to be the minimum multiple of $\prod\limits_{i=1}^tp_i$ such that for every pair $(i, q),\ i\in \{1,2,\ldots,t\},\ q \in D(N)$, at least one of the following conditions is satisfied:
\[
\begin{array}{l}
\\
\hspace{-7cm}\vspace{0.1cm}(1)\ q =p_i {\rm \ and}\  (p_i, b_i) \not= (2, 1),\\
\hspace{-7cm}\vspace{0.1cm}(2)\ b_i=c_i, \\
\hspace{-7cm}\vspace{0.1cm}(3)\ q \not= p_i {\rm \ and}\  q^{\text{ord}_{m_q} (q)}\not\equiv 1 \pmod {p_i^{b_i+1}}.
\end{array}
\]
\end{definition}

\begin{remark}
When $N=1$, we have $F(M,1)=1$.
\end{remark}

 We need the following important Schmidt's theorem.



\begin{theorem}\label{F(m,n)2}{ \rm \cite{S1999}}
Let $X\in \mathbb{Z}[\zeta_M]$ be of the form $$X=\sum\limits_{i=0}^{M-1}a_i\zeta_M^i$$
 with $0\leq a_i \leq C$ for some constant $C$ and assume that $X\cdot \overline{X}=N$ is an integer. Then
$$N \leq \frac{C^2\cdot F(M,N)^2}{4\cdot \varphi (F(M,N))},$$
where $\varphi$ is Euler function. 
\end{theorem}


By using Theorem \ref{F(m,n)2}, we obtain the following exponent bound.

\begin{theorem}\label{F(m,n)1}
Suppose there exists a $(v,m, k, \lambda)$-SEDF in $G$ with $m\geq 3$. Let $d$ be an integer with $t$ distinct prime divisors such that $2<d$ and $d\mid \text{exp}(G)$. Then
$$\lambda \leq \frac{(2^{t-1}v)^2\cdot F(d,\lambda)^2}{4d^2\cdot \varphi(F(d,\lambda))}.$$
\end{theorem}

\begin{proof} Let $\{D_1,\ldots, D_m\}$ be a $(v,m, k, \lambda)$-SEDF and set $D=\bigcup\limits_{i=1}^m D_i$. Since $d\mid \text{exp}(G)$, there exists a nonprincipal character $\chi$ of $G$ of order $d$. We distinguish two cases.

{\bf Case (i)} $\chi(D)=0$.

By Equation (\ref{Dj1}), we have $\chi(D_{j})\overline{\chi(D_{j})}=\lambda$. By Theorem \ref{F(m,n)2} and Corollary \ref{c1}, we obtain
$$\lambda \leq \frac{(2^{t-1}\frac{v}{d})^2\cdot F(d,\lambda)^2}{4\cdot \varphi(F(d,\lambda))}.$$

{\bf Case (ii)} $\chi(D)\neq 0$.

Set $\chi(D_{j})\overline{\chi(D_{j})}=A$. Since $m\geq 3$, by Lemma \ref{integer1} we have $A$ is a positive integer and $A\mid \lambda^2$. Then it holds that  $\chi(D-D_{j})\overline{\chi(D-D_{j})}=\frac{\lambda^2}{A}$. By Theorem \ref{F(m,n)2} and Corollary \ref{c1}, we obtain
$$A\leq \frac{(2^{t-1}\frac{v}{d})^2\cdot F(d,A)^2}{4\cdot \varphi(F(d,A))}\ {\rm and}\ \frac{\lambda^2}{A}\leq \frac{(2^{t-1}\frac{v}{d})^2\cdot F\left(d,\frac{\lambda^2}{A}\right)^2}{4\cdot \varphi\left(F\left(d,\frac{\lambda^2}{A}\right)\right)}.$$
Therefore, we obtain
$$\lambda \leq  \max\limits_{A,A\mid \lambda^2} \frac{(2^{t-1}v)^2\cdot F(d,A)F\left(d,\frac{\lambda^2}{A}\right)}{4d^2\sqrt{\varphi(F(d,A))\varphi\left(F\left(d,\frac{\lambda^2}{A}\right)\right)}}.$$
From the definition of $F(M,N)$, we know that $F(M,N')\mid F(M,N)$ for positive integer $N, N'$ with $N'\mid N$. It is easy to verify that $\frac{p^aN}{\sqrt{\varphi(p^aN)}}\leq \frac{p^bN}{\sqrt{\varphi(p^bN)}}$ for any prime $p$ and positive integers $N, a, b$ with $a\leq b$. Then we have 
$$\max\limits_{A,A\mid \lambda^2} \frac{(2^{t-1}v)^2\cdot F(d,A)F\left(d,\frac{\lambda^2}{A}\right)}{4d^2\sqrt{\varphi(F(d,A))\varphi\left(F\left(d,\frac{\lambda^2}{A}\right)\right)}}
=\frac{(2^{t-1}v)^2\cdot (F(d,\lambda))^2}{4d^2\varphi(F(d,\lambda))}.$$
Hence, 
$$\lambda \leq \frac{(2^{t-1}v)^2\cdot (F(d,\lambda))^2}{4d^2\cdot\varphi(F(d,\lambda))}.$$
This completes the proof. \qed
\end{proof}

By directly applying Theorem \ref{F(m,n)1}, we have obtained the following nonexistence results for SEDFs.
\begin{corollary}\label{FMN12}
Let $p$ be a prime and let $n, m,\lambda$ be three positive integers such that $m>2$. If $\lambda >\frac{p^2\cdot  (F(p^{n-1},\lambda))^2}{4\cdot\varphi(F(p^{n-1},\lambda))}$, then a $(p^n, m, k, \lambda)$-SEDF does not exist in any abelian group $G$ with $p^{n-1}\mid \text{exp}(G)$. 
\end{corollary}

We now illustrate the use of Theorem \ref{F(m,n)1} to rule out the existence of SEDFs in abelian groups.

\begin{example}
There does not exist a $(2401, 9, 180, 108)$-SEDF in any abelian group $G$ with $7^3 \mid \text{exp}(G)$. 
\end{example}
\begin{proof}
Suppose, for a contradiction, that there exists  a $(2401, 9, 180, 108)$-SEDF in an abelian group $G$ with $7^3 \mid \text{exp}(G)$.

Firstly, we compute $F(343, 108)$. By the definition of $F(M,N)$, we have $c_1=4, m_2=7$ and $m_3=7$. If $b_1 < 4$, then $(7, b_1) \neq (2, 1)$, $2^{\text{ord}_7(2)}=2^3\not \equiv 1 \pmod {7^{b_1+1}}$ and $3^{\text{ord}_7(3)}=3^6\not \equiv 1 \pmod {7^{b_1+1}}$. Thus $b_1=1$. So $F(343, 108)=7$.  

 Secondly, we compute
 $$\frac{(2^{t-1}v)^2\cdot (F(d,\lambda))^2}{4d^2\cdot \varphi(F(d,\lambda))}.$$
 Here $v=2401, d=343$, $t=1$ and $\lambda=108$. Then we have $F(d,\lambda)=7$. Applying Theorem \ref{F(m,n)1}, we have
 $$108=\lambda \leq  \frac{(2^{t-1}v)^2\cdot (F(d,\lambda))^2}{4d^2\cdot \varphi(F(d,\lambda))}=\frac{v^2 \times 7\times 7}{4 d^2\times \varphi(7)}=\frac{49v^2}{24d^2},$$
 that is, $v^2\geq 52.9d^2$, a contradiction. \qed
\end{proof}

\begin{example}
There does not exist a $(4096,14,210,140)$-SEDF in any abelian group $G$ with $2^{10}\mid \text{exp}(G)$. 
\end{example}
\begin{proof}
Suppose, for a contradiction, that there exists a $(4096,14,210,140)$-SEDF in an abelian group $G$ with $2^{10}\mid \text{exp}(G)$.

Firstly, we compute $F(1024, 140)$. By the definition of $F(M,N)$, we have $c_1=12, m_5=4, m_7=4$. If $b_1 < 12$, then we have $(2, b_1) \neq (2, 1)$,  $5^{\text{ord}_4(5)}=5 \not \equiv 1 \pmod {2^{b_1+1}}$ and $7^{\text{ord}_4(7)}=7^2\not \equiv 1 \pmod {2^{b_1+1}}$. Thus $b_1=4$. So $F(1024, 140)=2^4$. 

Secondly, we compute
$$\frac{(2^{t-1}v)^2\cdot (F(d,\lambda))^2}{4d^2\cdot \varphi(F(d,\lambda))}.$$
Here $v=4096, d=2^{10}$, $t=1$ and $\lambda=140$. Then we have $F(d,\lambda)=16$. Applying Theorem \ref{F(m,n)1}, we have
$$ \vspace{0.2cm} 140=\lambda \leq  \frac{(2^{t-1}v)^2\cdot (F(d,\lambda))^2}{4d^2\cdot \varphi(F(d,\lambda))} =\frac{v^2 \times 16\times 16}{4 d^2\times \varphi(16)}=\frac{8v^2}{d^2}, $$
that is, $v^2\geq 17.5d^2$, a contradiction. \qed
\end{proof}

\section{Conclusion}

In this paper, we used the results of decomposition of prime ideals, Schmidt's field descent method to obtain two exponent bounds. We also used the field descent method and Gauss sums to present some new necessary conditions for the existence of SEDFs, which yield new nonexistence results for certain types of SEDFs. These results generalize those  from \cite{BJWZ2018}, \cite{HP2018}, \cite{JL2018}, \cite{LLP2019} and \cite{MS-arXiv}.

Generalized SEDFs were also introduced in \cite{PS2016}. In 2018, Wen et al. \cite{WYFF2018} constructed a series of generalized SEDFs by using difference sets and cyclotomic classes. In the same year, Lu et al. \cite{LNC2018} gave the first recursive construction for generalized SEDFs and obtained some new generalized SEDFs for $m=2$. It is also interesting to study generalized SEDFs.


\end{document}